\documentclass{preprint}

\usepackage{units}

\usepackage{amsmath,amsfonts,mathtools}
\usepackage{dsfont,xfrac,sidecap,caption}
\usepackage{units}

\usepackage{graphicx}
\usepackage[usenames,dvipsnames]{xcolor}

\usepackage{subfigure}
\usepackage{booktabs}  % for better style of tablesl

\newcommand{\tolP}{\text{tol}_P}
\newcommand{\errP}{\text{err}_P}
\newcommand{\euler}{\mathrm{e}}
\newcommand{\uf}{u}
\newcommand{\vf}{\mathbf{v}}

\newcommand{\nt}{\vec{n}}
\newcommand{\taut}{\vec{\tau}}
\newcommand{\vt}{\mathbf{v}}
\newcommand{\ft}{\mathbf{f}}
\newcommand{\ut}{\mathbf{u}}
\newcommand{\wt}{\mathbf{w}}
\newcommand{\sigmat}{\boldsymbol{\sigma}}

\newcommand{\Ft}{\mathbf{F}}

\renewcommand{\div}{{\operatorname{div}}}

\renewcommand{\O}[1]{{\mathcal{O}}\left({#1}\right)}

\newcommand{\sfrei}[1]{{#1}}

\usepackage{tikz-cd}
\usetikzlibrary{calc,arrows,positioning}
\usepackage{pgfplots}
\usepgfplotslibrary{groupplots}
\pgfplotsset{compat=1.14}

\title{Efficient approximation of flow problems with multiple scales
  in time}

\author{S.~Frei\thanks{Department of Mathematics \& Statistics, University of Konstanz, \texttt{stefan.frei@uni-konstanz.de}} \and
  T.~Richter\thanks{Otto-von-Guericke-Universit\"at Magdeburg, 39106 
    Magdeburg, \texttt{thomas.richter@ovgu.de}, Interdisciplinary
    Center for Scientific Computing, Heidelberg}}

\begin{document}

\maketitle

\begin{abstract}
In this article we address flow  problems that carry a multiscale
character in time. In particular we consider the Navier-Stokes flow in a channel on a fast scale
that influences the movement of the boundary which undergoes a deformation on a slow scale
in time. We derive an averaging scheme that is of first order with
respect to the ratio of time-scales $\epsilon$.
% and of second order in
% all discretisation variables.
In order to cope with
the problem of unknown initial data for the fast scale problem, we
assume near-periodicity in time. 
Moreover, we construct a second-order accurate time discretisation scheme and derive a
complete error analysis for a corresponding simplified ODE system.
The resulting multiscale scheme does not ask for the continuous
simulation of the fast scale variable and shows powerful speed-ups up
to 1:10\,000 compared to a resolved simulation. 
Finally, we present some numerical examples for the full Navier-Stokes
system to illustrate the convergence and performance of the approach. 
\end{abstract}

%%%%%%%%%%%%%%%%%%%%%%%%%%%%%%%%%%%%%%%%%%%%%%%%%%

\section{Introduction}\label{sec:intro}

We are interested in the numerical approximation and long-term simulation of flow problems that carry a multiscale
character in time. Such problems appear for example in 
the formation of atherosclerotic plaque in arteries, where flow dynamics acting on a scale of milliseconds to seconds have an 
effect on plaque
growth in the vessel, which typically takes place within a range of several months. 
Another application is the investigation of chemical
flows in pipelines, where long-time effects of weathering, accelerated
by the transported substances, cause material alteration.

These examples have in common, that it is computationally infeasible
to resolve the fast scale over the whole time  
interval of interest. In the case of atherosclerotic plaque growth, a suitable time-step
of $\unit[\frac{1}{20}]{s}$ would require nearly $10^9$ steps to
cover the period of interest, which is at least $\unit[6]{months}$. 

Inspired by the temporal dynamics of atherosclerotic plaque growth, we
will consider the flow in a channel whose boundary is deformed over a
long time scale. This deformation is controlled by the
concentration variable $u(t)$ that is governed by a simple reaction equation
and that depends on the fluid-forces  
\begin{equation}\label{problem:coupled}
  \begin{aligned}
    \vt(0)&=\vt_0,&\quad \operatorname{div}\, \vt&=0,\quad
    \rho (\partial_t \vt + (\vt\cdot\nabla)\vt) - \div\,\sigmat(\vt,p) 
    =\ft
    &\quad \text{ in }&\Omega(u(t))\\
    u(0)&=0,&\quad u' &= \epsilon R(u,\vt).
  \end{aligned}
\end{equation}
Here, $\rho$ is the density of the fluid,
$\sigmat=\rho\nu(\nabla\vt+\nabla\vt^T)-pI$ the Cauchy stress with
the kinematic viscosity $\nu$ and $R(\vt, u)\ge 0$ a reaction term
describing the influence of the fluid forces (namely the
wall shear stress) on the boundary 
growth. {The growth term $R(\cdot,\cdot)$ it modeled
  such that $|R(\cdot,\cdot)| = {\cal O}(1)$:
  \begin{equation}\label{reaction}
    R(u,\vt)\coloneqq
    \big(1+u\big)^{-1}\big(1+|\sigma_{WSS}(\vt)|^2\big)^{-1},\quad
    \sigma_{WSS}(\vt)\coloneqq \sigma_0^{-1}\int_\Gamma\rho\nu
    \big(I_d-\nt\nt^T\big)(\nabla\vt+\nabla\vt^T)\nt\,\text{d}o,
  \end{equation}
  where $\nt$ denotes the outward facing unit normal
  vector at the boundary $\Gamma$. 
  The parameter $\sigma_0>0$ will be tuned to give
  $|\sigma_{WSS}(v)|={\cal O}(1)$, see Section~\ref{sec:num}.}
The domain $\Omega=\Omega(u(t))$ depends explicitly on the
concentration $u(t)$. We show a sketch of the configuration in
Figure~\ref{fig:config}. The flow problem is driven by a periodic oscillating
inflow profile of period $\unit[1]{s}$ 
\[
\vt=\vt^D\text{ on }\Gamma_{in},\text{ with }\vt^D(t)= \vt^D(t+1s).
\]
This period describes the \emph{fast scale} of the problem.
{By $\epsilon\ll 1$ we denote a small
  parameter that controls the time scale of the (slow) growth of the
  concentration, i.e. $|u'|={\cal O}(\epsilon)$ and $T={\cal
    O}(\epsilon^{-1})$ is the expected long term horizon. }
While the problem itself is strongly simplified compared to the
detailed non-linear mechano-chemical FSI model  
of plaque-growth~\cite{Chen2012, YangJaegerNeussRaduRichter2015,
  FreiRichterWick2016}, we choose the parameters in such a way that
the temporal dynamics are very similar.

  The structure of this article is as follows: \sfrei{In Section~\ref{sec:time}
  we introduce a simple model problem consisting of two coupled ODEs, that are related to
  \eqref{problem:coupled}, and for which we will be able to conduct a complete error analysis. Moreover, 
  we discuss some of the available approaches in literature and outline the multiscale algorithm developed in this article.
  In Section~\ref{sec:ms} we derive the effective
  long term equations and give an error analysis on the continuous level.} In Section~\ref{sec:timedisc} we describe the temporal
  discretisation of the multiscale scheme and show optimal order
  convergence in all discretisation parameters: mesh size $h$, time step
  size $k$ for the fast problem and time step size $K$ for the slowly
  evolving variable. In Section~\ref{sec:num} we apply the multiscale
  scheme to the complex problem introduced in Section~\ref{sec:time},
  which is based on the Navier-Stokes equations. We show numerically
  optimal-order convergence in agreement to the theoretical findings
  for the simplified system. We
  conclude with a short summary and a discussion of some open problems.

\begin{figure}[t]
  \begin{center}
    \includegraphics[width=\textwidth]{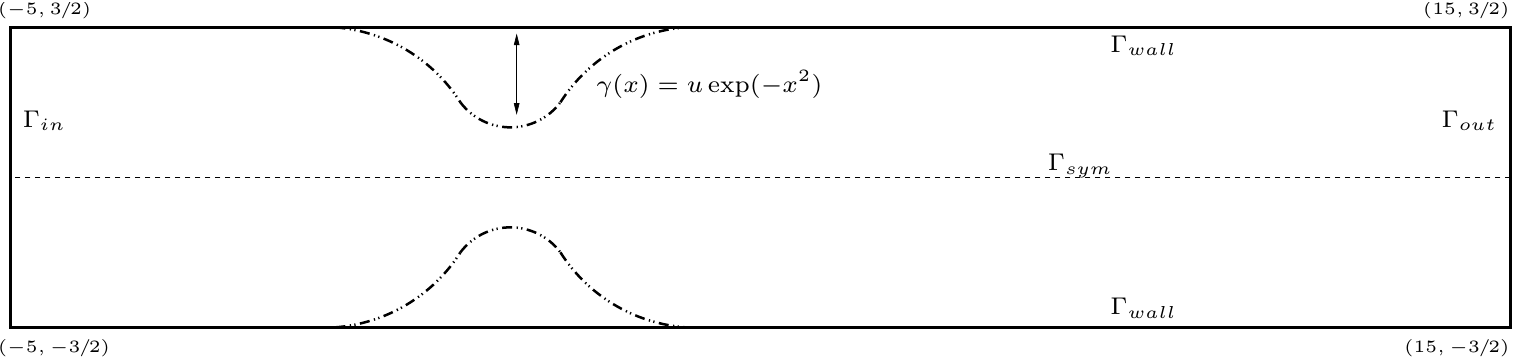}
  \end{center}
  \caption{Configuration of the test case. We study flow in a channel with a
    boundary $\Gamma$ that depends on a concentration variable
    $u$. This $u$ follows a simple reaction law with a right-hand side
    depending on the wall shear stress on
    $\Gamma_{wall}$.
  \label{fig:config}}
\end{figure}

%%%%%%%%%%%%%%%%%%%%%%%%%%%%%%%%%%%%%%%%%%%%%%%%%%
%%%%%%%%%%%%%%%%%%%%%%%%%%%%%%%%%%%%%%%%%%%%%%%%%%
%%%%%%%%%%%%%%%%%%%%%%%%%%%%%%%%%%%%%%%%%%%%%%%%%%

%\section{Problem description} [47] Majda

\section{Time scales}\label{sec:time}

In this section we start by analyzing the temporal multiscale
character of the plaque formation problem. We simplify the
coupled problem and introduce a model problem coupling two
ODEs. Then, we present various approaches for the numerical treatment
of temporal multiscale problems that are discussed in
literature. Finally we sketch the idea of the multiscale scheme that
is considered in this work, which fits into the framework of the
\sfrei{heterogeneous} multiscale method.

\subsection{A model problem}
{
  To start with, we introduce a simple model problem, a system of two
  ODEs that shows the same coupling and temporal multiscale
  characteristics as the full plaque growth system
  \sfrei{
  \begin{subequations}
  \label{problem:simple}
  \begin{align}
    \label{motms:kappa}
    u(0)&=u_0,&\quad  u'(t) &= \epsilon R(u(t),v(t))\\
    \label{motms:v}
    v(0)&=v_0,&\quad v'(t) + \lambda(u(t))v(t) &= f(t),
  \end{align}
\end{subequations}}
% where $f$ is periodic with period $1$ and
% $\lambda(u)\ge \lambda_0>0$.
%   
%   %
%   \begin{equation}\label{problem:simple}
%     v(0)=v_0,\; u(0)=u_0,\quad
%     v' + \lambda(u)v = f,\;
%     u' =\epsilon R(u,v),
%   \end{equation}
%   %
  where $f(t)=f(t+1)$ is
  periodic and $R(\cdot,\cdot)$ is given by
  \begin{equation}\label{reaction:simple}
    R(u,v)\coloneqq \big(1+u\big)^{-1}\big(1+v^2\big)^{-1}. 
  \end{equation}
  For the reaction term it holds $|R(u,v)|\le 1$. \sfrei{The parameter $\lambda(u)\ge \lambda_0>0$ depends on the concentration $u$. 
  We will assume that the relation $u\to\lambda(u)$ is differentiable and that the derivative $\frac{d\lambda(u)}{du}$ remains bounded.}}  
  
 {In
  Section~\ref{sec:relevance} we will argue that this system can indeed be
  considered as a simplification of the full plaque growth system by
  neglecting the nonlinearity and by diagonalizing the
  resulting Stokes equation with respect to an orthonormal
  eigenfunction basis.
Moreover, if we introduce $\tau\coloneqq \epsilon t$, 
  $v_\tau(\tau)\coloneqq v(t)$,
  $u_\tau(\tau)\coloneqq u(t)$ as well as
  $f_\tau(\tau)\coloneqq f(t)$
  we can scale this system  to
  \begin{equation}\label{problem:simple:scaled}
    v_\tau(0)=v_0,\; u_\tau(0)=u_0,\quad 
    v_\tau'+\epsilon^{-1}\lambda(u_\tau)
    v_\tau  =\epsilon^{-1}f_\tau,\; 
    u_\tau' =R(u_\tau,v_\tau),
  \end{equation}
  which reveals the typical character of ODE systems with multiple
  scales in time as discussed in~\cite{E2011,Abdulleetal2012}. In the
  language of the heterogeneous multiscale method (HMM), see
  also~\cite{EEngquist2003}, such a problem is called a type B problem
  and it is characterized by the acting of fast and slow scales
  throughout the whole (long) time span $[0,T]$ in contrast to
  problems with localised singular behavior. 
  Since $|R|={\cal O}(1)$ it holds $|u'_\tau|={\cal O}(1)$ and
  $u_\tau$ describes the slow variable while $|v'_\tau|={\cal
    O}(\epsilon^{-1})$ indicates the fast and oscillatory
  variable.}

\subsection{Numerical approaches for temporal multiscale problems}

While multiscale problems in space are extensively studied in
literature, see e.g.\,\cite{Cioranescu, Oleinik}, less 
works are found on problems with multiscale character in time. Some literature exists that uses a homogenisation approach based on asymptotic expansions in 
time for viscoelastic, viscoplastic or elasto-viscoplastic solids~\cite{Guennouni1988, YuFish2002, AubryPuel2010, HaoualaDoghri2015}. Under suitable 
assumptions, the short-scale part of the multiscale algorithm becomes stationary for this class of equations, such that difficulties to 
define initial values on the short scale are avoided.

{Multirate time stepping
  methods~\cite{GanderHalpern2013} split the system into slow and
  large components and use different time step sizes according to the
  dominant scales. All scales are still resolved on the complete time
  interval. Since the fast scale of problems~(\ref{problem:coupled})
  which requires a small time step
  is the computationally intensive Navier-Stokes equations and since
  the scales are vastly 
  separated, such multirate methods are not
  appropriate for the problem under investigation.}

{In the context of continuum damage mechanics,
  processes with high frequent oscillatory impact can be approached by
  block cycle jumping techniques~\cite{LemaitreDoghri1994}, where a
  large number of  cycles is skipped and replaced by linear
  approximation of the damage effect.
  An overview of different techniques is given in the first two
  introductionary sections of~\cite{ChakrabortyGhosh2013}. These
  approaches do not resolve the complete system on full temporal
  interval but reside on local solutions. This gives rise to the
  problem of finding initial values. }

  If the time scales are close enough that the short-scale dynamics
  can be resolved within one time step of the long-scale
  discretisation,  the \emph{Variational Multiscale
    Method}~\cite{HughesStewart1996, Bottasso2002} or approaches that
  construct long-scale basis functions from the short-scale
  information~\cite{Passieuxetal2010, Ammaretal2012} are
  applicable. Similar algorithms are also used to construct
  parallel-in time integrators, for  example the \emph{parareal}
  method~\cite{MadayTurinici2005}. In this work, we are interested in
  problems with a stronger \emph{scale separation}, where the
  resolution of the short scale within a long-scale interval is very
  costly up to computationally unfeasible.

  Only very few numerical works can be found concerning flow problems with multiple scales in time. An exception are the 
works of Masud \& Khurram~\cite{MasudKhurram2006, MasudKhurram2006FSI}, where the \emph{Variational Multiscale Method} is applied,
assuming again that the time scales are sufficiently close. On the other hand, several theoretical works exist that 
show convergence towards averaged equations for specific flow configurations in 
the situation that the ratio of 
time scales $\epsilon = \frac{t_\text{fast}}{T_\text{long}}$ tends to zero, 
see e.g.~\cite{Ilyin1998, Chepyzhov2008, Levenshtam2015}, however without considering practical 
numerical algorithms or discretisation. 

A common numerical approach is to replace the fast problem by an
averaged one using a fixed-in-time inflow
profile~\cite{YangJaegerNeussRaduRichter2015, Chen2012}. It is however
widely accepted and also confirmed in numerical
studies~\cite{FreiRichterWick2016} that such a simple averaging does not necessarily 
reproduce the correct dynamics.
In~\cite{FreiRichterWick2016} we presented a
first multiscale scheme for the approximation of such a problem, however
with a focus on the modelling of a full closure of the channel and
without any analysis on the robustness and accuracy. Similar algorithms can be found in 
Sanders et al~\cite{SandersVerhulstMurdock2007} and by Crouch \& Oskay~\cite{CrouchOskay2015}
in different applications.
In this work, we will
derive an improved algorithm in a mathematically rigorous way, including a 
detailed error analysis for both modelling and discretisation errors. To our knowledge this 
is the first time that the interplay between modelling errors of the temporal multiscale scheme 
and temporal discretisation errors on both scales is analysed.

{One of the most prominent class of techniques is
    the \emph{heterogeneous multiscale method
      (HMM)}~\cite{EEngquist2003,E2011,Abdulleetal2012,EngquistTsai2005} that aims 
    at an efficient decoupling of macro-scale and micro-scale, where the
    latter one enters the macro-scale problem in terms of temporal
    averages. Typically, the procedure is as follows: one determines
    the fast and the slow variables of the coupled problem. For the
    slow variables an integrator with a long time step $\Delta
    T\coloneqq T_{n+1}-T_n$ and good stability properties is used. At each of
    these macro time steps, the fast scale problem is initialized
    based on the current slow variable and solved on the interval
    $I_n^\eta\coloneqq [T_n,T_n+\eta]$. Finally, the fast variable
    output on $I_n^\eta$ is averaged to yield the effective operator
    for the slow scale problem. }

  {The efficiency of the resulting HMM scheme depends on
    the choice of $\eta$ which indicates the scale to allow for
    equilibration and adjustment of the micro model. Too large  values
    will reduce the 
    efficiency, too small values will limit the accuracy. The
    underlying problem is the lack of initial values at the new macro
    step for the fast scale, which in our case~(\ref{problem:coupled})
    or~(\ref{problem:simple}), is the oscillatory velocity
    $\vt(t)$ and $v(t)$, respectively. The realisation presented
    in this article is based on time-periodic solutions to the
    micro problem. Instead of solving the microscale problem on an
    interval $I_n^\eta$ at macro step $T_n$ we aim at a localised
    solution of the fluid problem that satisfies a periodicity
    condition in time. This approach allows us to conduct a
    complete error analysis of the resulting scheme when applied to
    the simplified model problem~(\ref{problem:simple}). 
  } 

  \subsection{Outline of the multiscale scheme}
 
  {We conclude this section by briefly describing the
    multiscale algorithm that is considered in this article. The
    derivation given here is based on problem~(\ref{problem:simple}). We start
    by defining the slow variable as average of the concentration $\uf(t)$
    \begin{equation}\label{show:0}
      U(t)\coloneqq \int_t^{t+\unit[1]{s}}\uf(s)\,\text{d}s. 
    \end{equation}
    This gives rise to the averaged equation for the long term
    dynamics
    \begin{equation}\label{show:1}
      U'(t) =  \int_t^{t+1}\epsilon R(\uf(s),v(s))\,\text{d}s 
    \end{equation}
    A time integration formula with a macro time-step is used to
    approximate this equation. Two approximation steps are performed to
    reach an effective equation. First, the reaction term
    in~(\ref{show:1}) is evaluated in $U(t)$ instead of $\uf(s)$ and
    second, the fast component 
    $v(s)$ will be replaced by the localised solution of the
    time-periodic problem obtained for a fixed value of $U(t)$:
    \begin{equation}\label{show:2}
      v'_{U(t)} + \lambda(U(t)) v_{U(t)} = f\text{ in }
      [0,1]\text{ with } 
      v_{U(t)}(1)=v_{U(t)}(0). 
    \end{equation}
    \sfrei{These approximations will be discussed and analysed in the following section.}    
    Assuming that~(\ref{show:1}), approximated in these two steps, is
    integrated with the forward Euler 
    method, a macro time step is given by
    \begin{equation}\label{show:3}
      U_{n}= U_{n-1} + (T_{n}-T_{n-1})  \int_{T_{n-1}}^{T_n}
      \epsilon R(U_{n-1},v_{U_{n-1}}(s))\,\text{d}s. 
    \end{equation}}
    
    \sfrei{
    \subsubsection{Motivation for the locally periodic approximations}
    
    Due to the nonlinearity of the reaction term $R(\cdot,\cdot)$, the micro-scale variations
    in $v(s)$ can not simply be averaged. Instead the velocities $v(s)$ need to be computed
    on the fast scale in each macro step $T_{n-1} \to T_{n}$, in order to obtain a good approximation of 
    the integral on the right-hand side of \eqref{show:1}. A computation of $v(s)$ over the complete 
    interval $[T_{n-1}, T_{n}]$ is however unfeasible for small $\epsilon$. For this reason, the imposition 
    of accurate initial values $v(T_n)$ for the fast-scale problem is not straight-forward. Neither 
    the short-scale velocity $v(T_{n-1})$ from the previous macro-time step nor an averaged quantity 
    $V(T_n)$ can guarantee a sufficiently good approximation for $v(T_n)$. In practice, a relaxation 
    time $\eta$ is frequently introduced (see for example~\cite{EngquistTsai2005, Abdulleetal2012}), in order to improve the initial values by means of a few 
    forward iterations.
    
    As an alternative, we propose to introduce the time-periodic fast-scale problem \eqref{show:2}. This has the advantage that in principle only one 
  period of the fast-scale problem needs to be resolved per macro-step. We can show theoretically (Lemma 8) that the approximation error introduced by the periodic problem
  is of order $\epsilon$. Efficient approximations of these time-periodic problems will be discussed in Section~\ref{sec:periodic}.}

%     By introducing time-periodic solutions $v_{U_n}$ we avoid the
%     necessity of choosing the relaxation time $\eta$. In principle it
%     is sufficient to integrate the local problems~(\ref{show:2}) over
%     one single interval for each macro time step. This would however
%     require exact initial values $v_{U_n}(0)=v_{U_n}(1)$.  
%     In
%     Section~\ref{sec:periodic} we discuss the approximation of these
%     periodic problems. 

\sfrei{
\subsubsection{Abstract multiscale scheme}

We conclude this section by formulating the abstract multiscale scheme that can be
applied to both problems, the plaque growth system and the simplified
model problem. }
{
\begin{algo}[Abstract Multiscale Scheme]\label{algo:scheme}
  Let $0=T_0<T_1<\dots<T_N=T$ be a partition of the macro
  interval with uniform step size $K\coloneqq T_n-T_{n-1}$. Further,
  let $U_0\coloneqq \uf_0$ be the initial value 
  of the slow variable. Iterate for $n=1,2,\dots$
  \begin{enumerate}
  \item Solve the time-periodic problem~(\ref{periodic:simple})
    or~(\ref{periodic}) for $\vt_{U_{n-1}}$. 
  \item Evaluate the reaction term
    \[
    R_{n-1}\coloneqq \int_0^1
    R(U_{n-1},\vt_{U_{n-1}}(s))\,\text{d}s 
    \]
  \item Forward the slow variable with an (explicit) one-step 
    scheme
    \[
    U_n = {\cal F}(K;U_{n-1};R_{n-1})
    \]
  \end{enumerate}
\end{algo}
\sfrei{The structure of the time integrator ${\cal F}$ depends on the system of equations and 
the time-stepping method. For simplicity,} we have formulated the algorithm for an explicit time integrator in
the slow variable. 
The use of an implicit time stepping scheme would
require an iterated evaluation of the periodic problem for updated
values of $U_n$. In the case of $r$-step schemes periodic solutions $\vt_{U_{n-k}}$ 
would be required for $k=1,\dots,r$.
}

\begin{remark}
\sfrei{We note that the proposed algorithm does not compute an average
  of the fast variable $v$. An approximation to $v$ is only computed
  on the short periodic interval of the fast-scale as $\vt_{U_{n-1}}$ (Step 1). The 
slow variable $u$, on the other hand, is only computed on the slow scale as average $U$ (Step 3).}
\end{remark}

\sfrei{The typical behaviour of the slow and fast variables is illustrated in Figure~\ref{fig:scheme} (top). On the bottom of Figure~\ref{fig:scheme}, the 
multiscale algorithm is visualised, including the transfer of quantities between the slow and the fast scale. For ease of 
presentation a large $\epsilon$ has been chosen for the purpose of visualisation.}

\begin{figure}[h!]
  \begin{center}
    \includegraphics[width=0.7\textwidth]{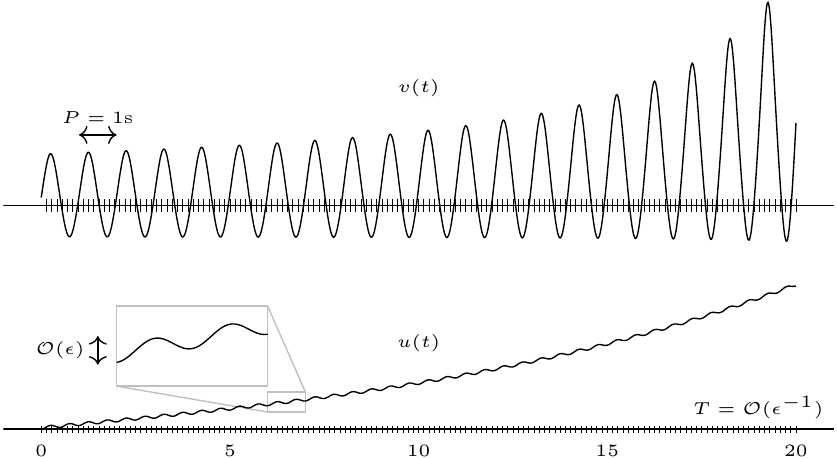}
    \caption*{{Layout of the multiscale problem. The
        fast variable 
      $v(t)$ (top) and the ``slow'' variable $u(t)$ (bottom) couple on the
      complete long time interval $I=[0,T]$. The ``slow'' variable is
      also oscillating, but the oscillations are small, of size ${\cal
        O}(\epsilon)$. The fast variable is locally nearly periodic. }}
    \includegraphics[width=0.7\textwidth]{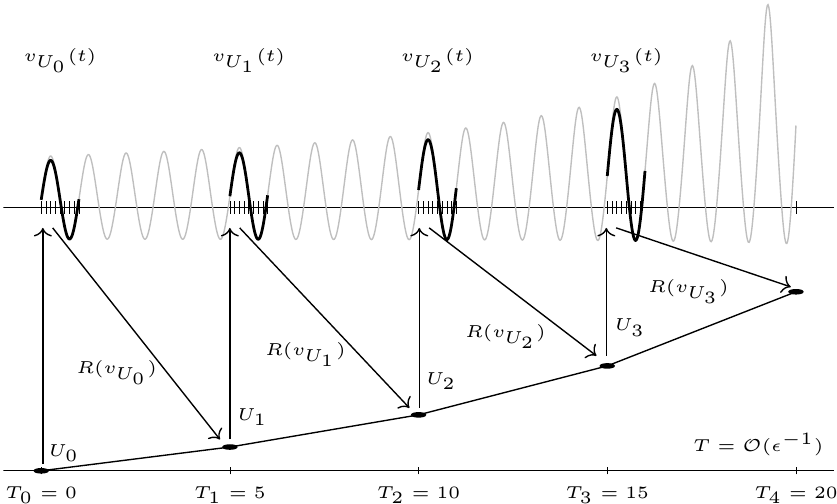}
    \caption{\label{fig:scheme}{Construction of the multiscale scheme: 0. The slow variable $U(t)$ is discretised
      with a time-stepping scheme with macro step size $K\gg
      1$.
      1. In each step the current slow state $U_{n-1}$
      is transferred to the fast problem (top) and a periodic solution $\vt_{U_{n-1}}$ is 
      computed on $[T_{n-1},T_{n-1}+1]$ as approximation for $\vt(t)$.      
    %1. We average
    %  the slow variable $U(t)$. 
%       2. The slow variable is discretized
%       with a time stepping scheme of large step size $K\ll
%       \epsilon$. 3. In each step, the current slow state $U_n$ is
%       transferred to the fast problem and a periodic solution is
%       computed on $[T_n,T_n+1]$. 
      2. The averaged reaction term $R_{n-1}=R(\vt_{U_{n-1}})$ is computed from $\vt_{U_{n-1}}$ 
      and transferred to the slow problem (bottom).
      3. The slow variable $U_{n}$ is updated by the macro step $T_{n-1}\to T_n$.}}
  \end{center}
\end{figure}
    
    {To conclude this section we anticipate the main result of 
    the analysis given below. For the combination of the second order
    Adams-Bashforth rule for the discretisation of the slow problem and the Crank-Nicolson
    scheme for the fast problem we will show optimal convergence of  
    the resulting multiscale scheme:
    \[
    |U_N-\uf(T)| = {\cal O}(\epsilon) + {\cal O}(\epsilon^2 K^2) +
    {\cal O}(k^2) 
    + {\cal O}(\tolP).
    \]
    By $K$ we denote the step size of the macro solver,
    $k$ is the step size of the  micro solver and by 
    $\tolP$ we denote the tolerance of the periodicity constraint:
    $\max_n |v_{U_n}(1)-v_{U_n}(0)|<\tolP$.}

  %\section{Temporal multiscale algorithms}
  \section{Derivation and analysis of the effective equations}
  \label{sec:ms}

  {
  In this section we derive the temporal multiscale scheme that has been
  outlined in the previous section. We will discuss the coupled
  Navier-Stokes problem on the evolving domain $\Omega(u)$,
  problem~(\ref{problem:coupled}) and the reduced ODE
  system~(\ref{problem:simple}) side by side. Whenever 
  results apply to the ODE system only, we will clearly mention
  this. We start by collecting some preliminary assumptions on the
  underlying problems. }

\subsection{Preliminaries}
\begin{assumption}[Reaction term]\label{ass:reaction}
  Let $u_{max}<\infty$ be a maximum concentration. 
  Let $0\le u\le u_{max}$ and $\vt\in X$ ($X=\mathds{R}$ for the model 
  problem, $X=H^2(\Omega)$ for the plaque growth problem).
  The reaction term is bounded
  \begin{equation}\label{ass:reaction:1}
    |R(u,\vt)| \le C_{A\ref{ass:reaction}a},
  \end{equation}
  and Lipschitz continuous in both arguments
  \begin{equation}\label{ass:reaction:2}
    |R(u_1,\vf)-R(u_2,\vf)| \le C_{A\ref{ass:reaction}b} |u_1-u_2|,\quad
    |R(u,\vf_1)-R(u,\vf_2)| \le C_{A\ref{ass:reaction}b} \|\vf_1-\vf_2\|_X,
  \end{equation}
  where the constant $C$ does not depend on $\epsilon$. 
\end{assumption}

{
Assumption~\ref{ass:reaction} is easily
verified for the simplified reaction term~(\ref{reaction:simple}). A proof for
Navier-Stokes case will be given in Section~\ref{sec:relevance}. }

{\begin{remark}[Generic constants]
  Throughout this manuscript we use generic constants $C$. These
  constants may depend on the domain $\Omega$, the maximum concentration
  $u_{max}$, the right hand side $\ft$ and the Dirichlet data.
  They do, however, not depend on the solution, the scale parameter
  $\epsilon$ or the discretisation parameters that will be introduced
  in the remainder of this article. 
\end{remark}}
We further assume that the isolated micro problems allow for a unique
periodic solution:
\begin{assumption}[Periodic solution]
  \label{ass:periodic}
  Let $\uf\in\mathds{R}$ with  $0\le \uf\le \uf_{max}$. We assume that
  there exist unique periodic solutions $v_\uf\in C([0,1])$ to 
  \begin{equation}\label{periodic:simple}
        \partial_t v_\uf + \lambda(\uf)v_\uf= f\text{ in
    }[0,1],\quad 
    v_\uf(1)=v_\uf(0),
  \end{equation}
  as well as solutions $\vt_\uf\in H^2(\Omega(u), p_\uf \in H^1(\Omega(u))$ to the incompressible Navier-Stokes
  equations
  \begin{subequations}\label{periodic}
    \begin{align}
      \nabla\cdot\vt_\uf = 0,\quad
      \rho(\partial_t \vt_\uf + (\vt_\uf\cdot\nabla)\vt_\uf)
      - \div\,\sigmat(\vt_\uf,p_\uf) &= \ft&&\text{in }[0,1]\times
      \Omega(u)\\
      \vt_\uf &= \vt^D_\uf&&\text{on }[0,1]\times \partial\Omega(\uf)\\
      \vt_\uf(1)&=\vt_\uf(0)&&\text{in }\Omega(\uf)
    \end{align}
  \end{subequations}
  Both solutions are uniformly bounded in time
  \begin{equation}\label{boundflowperiodic}
    \sup_{t\in [0,T]}|v_\uf(t)| \le C,\quad 
    \sup_{t\in [0,T]}\Big(\|\vt_\uf(t)\|_{H^2(\Omega)} +
    \|p_{\uf}(t)\|_{H^1(\Omega)}\Big) \le C. 
  \end{equation}
\end{assumption}
For the ODE problem~(\ref{periodic:simple}), the existence of a unique
periodic solutions follows from the evolution of $w(t)=v(t+1)-v(t)$
which fulfills $w'+\lambda(u)w=0$ and thus vanishes, see~\cite{Richter2020}. 
For a discussion on the Navier-Stokes equations we refer to
Section~\ref{sec:relevance}.

\subsection{Derivation of an effective equation}

We introduce the averaged concentration
\[
U(t) \coloneqq \int_{t}^{t+1} \uf(s)\,\text{d}s.
\]
Using (\ref{problem:coupled}) and (\ref{problem:simple}), respectively, 
and inserting $\pm R\big(U(t),\vt(s)\big)$, we have
\[
U'(t)=\int_{t}^{t+1}\epsilon R\big(\uf(s),\vt(s)\big)\,\text{d}s
= \int_{t}^{t+1}\epsilon R\big(U(t),\vt(s)\big)\,\text{d}s
- \int_{t}^{t+1}\epsilon\left(R\big(U(t),\vt(s)\big) -
R\big(\uf(s),\vt(s)\big)\right)\, \text{d}s. 
\]

\begin{lemma}[Averaging error]\label{lemma:average}
  Let $u\in C^1([0,T])$, $\vt\in C([0,T];X)$ and let
  Assumption~\ref{ass:reaction} be satisfied. Then, it holds
  \[
  \max_t \Big| \int_{t}^{t+1} \left(
  R(U(t), \vt(s))-R(u(s),\vt(s))
  \right)\,\text{d}s\Big|\le C \epsilon,
  \]
  with a constant $C>0$ that depends on
  Assumption~\ref{ass:reaction}. 
\end{lemma}
\begin{proof}
  By Lipschitz continuity of $R(\cdot,\cdot)$ in the first
  argument~(\ref{ass:reaction:2}) it holds
  \begin{equation}\label{R1}
    |R(U(t),\vt(s))-R(\uf(s),\vt(s))|
    \le C | \uf(s)-U(t)|.
  \end{equation}
  We estimate
  \begin{equation}\label{R2}
    \int_{t}^{t+1} |\uf(s)-U(t)|\,\text{d}s 
  =\int_{t}^{t+1}
  \left|\int_t^{t+1}(\uf(s)-\uf(r))\,\text{d}r\right|\,\text{d}s
  =\int_t^{t+1}
  \left|\int_t^{t+1}\int_r^s
  \uf'(x)\,\text{d}x\,\text{d}r\right|\,\text{d}s 
  \le C\epsilon,
  \end{equation}
  where we used~(\ref{ass:reaction:1}), 
  such that a combination of \eqref{R1} and \eqref{R2} shows the assertion. 
\end{proof}

\noindent We can thus approximate the averaged evolution equation for
$U$ by 
\begin{equation}\label{barkappa:1}
  U'(t) = \int_t^{t+1}\epsilon R(U(t),\vt(s))\,\text{d}s + 
  \O{\epsilon^2}.  
\end{equation}
%
% Due to the nonlinearity of $R(\cdot,\cdot)$, the micro scale variations
% in $\vt(s)$ can not simply be averaged. Moreover, equation~(\ref{barkappa:1}) can not be used
% for an efficient long time-step discretisation $T_{n-1}\to
% T_{n}=T_{n-1}+K$, as it would require knowledge of the solution
% $\vt(T_{n})$ which dynamically evolves from $\vt(T_{n-1})$ in $K/k$
% micro steps.
% This is the fundamental dilemma of
% temporal multiscale methods. At the heart of this problem is the lack of
% initial values for a local reconstruction of the fast scale problem at
% a given discrete long time step $T_{n-1}\to T_{n}$. 
% To overcome this problem we introduce the time-periodic solution
% $(\vt_U,p_U)$ to the Navier-Stokes equation for a fixed
% value $U=U(t)$ of the slow variable, see (\ref{periodic}),
% and~(\ref{periodic:simple}) in the case of the ODE system respectively. 

The benefit of introducing the periodic solution lies in a localisation of the
fast scale influences. Given an approximation $U_n$ at time $T_n$, the
micro scale influence $(\vt_{U_n},p_{U_n})$ can
be determined independent of the last approximation
$(\vt_{U_{n-1}},p_{U_{n-1}})$. 
We approximate the averaged equation~(\ref{barkappa:1}) by inserting
the periodic solution $\vt_{U(t)}(s)$ for a fixed value $U(t)$ 
\begin{equation}\label{barkappa:2}
  U'(t) = \int_t^{t+1}
  \epsilon R(U(t),\vt_{U(t)}(s))\,\text{d}s
  + \int_t^{t+1}
  \epsilon\big(R(U(t),\vt(s))-
  R(U(t),\vt_{U(t)}(s))\big)\,\text{d}s + \O{\epsilon^2}.
\end{equation}
We will show that the second remainder 
\begin{equation}\label{barkappa:bound}
  \max_t \Big| \int_t^{t+1}\epsilon \big(R(U(t),\vt(s))-
  R(U(t),\vt_{U(t)}(s))\big)\,\text{d}s\Big|
  %= \O{\epsilon^2}
\end{equation}
is also of order $\O{\epsilon^2}$. \sfrei{The analysis is presented
for the ODE system \eqref{problem:simple} in the following section (Lemma~\ref{lemma:periodic},~\ref{lemma:approxv}
and~\ref{lemma:conterror}). Extensions to the full system \eqref{problem:coupled} are discussed in Section~\ref{sec:num}.}
Having shown that the average $U(t)$ satisfies the equation
\begin{equation}\label{barkappa:2.5}
  U'(t) = \int_t^{t+1}
  \epsilon R\big(U(t),\vt_{U(t)}(s)\big)\,\text{d}s + \O{\epsilon^2}
\end{equation}
we define the effective equation for the approximation of $U(t)$ by
neglecting the remainder of order ${\cal O}(\epsilon^2)$, i.e. by the
equation 
\begin{equation}\label{barkappa:3}
  U'(t) = \int_t^{t+1}\epsilon
  R\big(U(t),\vt_{U(t)}(s)\big)\,\text{d}s,\quad
  U(0) = u_0. 
\end{equation}
%
%Strictly speaking,~(\ref{barkappa:2.5}) and~(\ref{barkappa:3}) define
%different functions $U(t)$. 
In Lemma~\ref{lemma:conterror}, we 
will estimate the error resulting from skipping the remainder
${\cal O}(\epsilon^2)$ in~(\ref{barkappa:2.5}). 
We further note that the initial values $U(0)=u_0$ and the averaged
initial $\int_0^1 u(s)\,\text{d}s$  do not necessarily 
coincide. Instead,~(\ref{barkappa:3}) deals with an offset of order
${\cal O}(\epsilon)$:
\begin{equation}\label{initialerror}
  \int_0^1 u(t)\,\text{d}t =
  \int_0^1 u(0)+\int_0^s u'(s)\,\text{d}s\,\text{d}t
  = u_0 + \O{\epsilon}. 
\end{equation}

\subsection{Analysis \sfrei{of the averaging error for} the model problem}\label{analysis}

In this section, we outline the ideas for showing convergence of the
multiscale scheme, Algorithm~\ref{algo:scheme}. As mentioned above the
analysis for the  
Navier-Stokes/ODE system is beyond the scope of this work. Instead we consider problem~(\ref{problem:simple}). 
One reason
is the lack of unique periodic solutions
$(\vt_{U(t)}(s),p_{U(t)}(s))$ for larger Reynolds numbers. Second,
the following analysis is based on the linearity of the model
problem. 
% , which we
% repeat for the sake of exposition: \sfrei{Ist das notwendig?} 
%
% \begin{subequations}
%   \begin{align}
%     \label{motms:kappa}
%     u(0)&=u_0,&\quad  u'(t) &= \epsilon R(u(t),v(t))\\
%     \label{motms:v}
%     v(0)&=v_0,&\quad v'(t) + \lambda(u(t))v(t) &= f(t),
%   \end{align}
% \end{subequations} 
% where $f$ is periodic with period $1$ and
% $\lambda(u)\ge \lambda_0>0$.
In addition to Assumptions~\ref{ass:reaction}
and~\ref{ass:periodic} we assume:
\begin{assumption}\label{ass:diff}
  We assume that the map $u\mapsto
  \lambda(u)$  is differentiable with a bounded derivative
  \begin{equation}\label{mot:3}
    \Big|\frac{d\lambda(u)}{d u}\Big|\le C_{A\ref{ass:diff}},\quad u\in
              [0,u_{max}].
  \end{equation}
\end{assumption}

Algorithm~\ref{algo:scheme} applied to the model problem (\ref{problem:simple}) calls for the
solution of the following averaged slow problem
\begin{equation}\label{motms:slow}
  U'(t) = \int_t^{t+1}\epsilon
  R\big(U(t),v_{U(t)}(s)\big)\,\text{d}s,\quad U(0)=u_0,
\end{equation}
and the corresponding time-periodic micro problems
\begin{equation}\label{motms:periodic}
  v_{U}'(t)+\lambda(U)v_{U}(t) =f(t),\quad
  v_U(1)=v_U(0),
\end{equation}
for each fixed parameter $0\le U\le u_{max}$. 
\begin{lemma}[Periodic solutions]\label{lemma:periodic}
  Let $0\le u\le u_{max}$ be fixed and let
  Assumption~\ref{ass:diff} hold. For the solution to the periodic
  problem~(\ref{motms:periodic}) it holds
  \begin{equation}\label{periodic:1}
    |v_u(t)| \le C_{L\ref{lemma:periodic}a}.
  \end{equation}
  Further, for $0\le u,\eta\le u_{max}$ let
  $v_u(t),v_\eta(t)$ be two such periodic solutions. It holds
  \begin{equation}\label{periodic:2}
    |v_u(t)-v_\eta(t)| \le C_{L\ref{lemma:periodic}b}
    |\lambda(u)-\lambda(\eta)|,
  \end{equation}
  where $C$ depends on $f$ and $\lambda_0>0$. 
\end{lemma}
\begin{proof}
  \emph{(i)} To show~(\ref{periodic:1}) we skip the index $u$ for better readability. 
  The general solution to the ODE is given by
  \begin{equation}
    \label{solution}
    v(t) = \euler^{-\lambda t}
    \Big(v(0) +  \int_0^t
    f(s)\euler^{\lambda s}\,\text{d}s\Big),
  \end{equation}
  which we estimate by
  \begin{equation}\label{p0}
    |v(t)|\le \euler^{-\lambda t} |v_0|+\frac{1}{\lambda}
    \|f\|_{L^\infty([0,1])}. 
  \end{equation}
  Since $v(t)$ is periodic, $v(1)=v(0)$, we obtain by~(\ref{solution})
  \begin{equation}\label{p1}
    v(0)=\frac{\euler^{-\lambda}}{1-\euler^{-\lambda}}
    \int_0^1 f(s)\euler^{\lambda s}\,\text{d}s
    \quad\Rightarrow\quad
    |v(0)|\le \frac{1}{\lambda}    \|f\|_{L^\infty([0,1])}
  \end{equation}
  Inserting~(\ref{p1}) into~\eqref{p0} we get for all $t\in [0,1]$
  \begin{equation}\label{p1.5}
    |v(t)| \le
    \frac{1+\euler^{-\lambda t}}{\lambda}
    \|f\|_{L^\infty([0,1])} \le
    \frac{2}{\lambda} \|f\|_{L^\infty([0,1])}, 
  \end{equation}
  which gives~(\ref{periodic:1}) since $\lambda \ge \lambda_0$. 

  \noindent\emph{(ii)} Let $w(t):=v_u(t)-v_\eta(t)$. It holds
  \[
  w'(t) + \lambda(u)w(t) = \big(\lambda(\eta)-\lambda(u)\big)
  v_\eta(t),\quad
  w(1)=w(0)=v_u(0)-v_\eta(0). 
  \]
  Note that the right-hand side of this ODE is time-periodic. Hence,
  we use~(\ref{p1.5}) twice and obtain the estimate
  \[
  |w(t)| \le \frac{2}{\lambda(u)}|\lambda(\eta)-\lambda(u)|
  \max_{t\in[0,1]} |v_\eta(t)|
  \le \frac{4}{\lambda(u)\lambda(\eta)}|\lambda(\eta)-\lambda(u)|
  \|f\|_{L^\infty([0,1])}.
  \]
\end{proof}

{The following essential lemma sets the foundation for
  replacing the dynamic fast component $v(t)$ by localised periodic in
  time solutions. For a given slow function $u(t)$ we will compare the
  corresponding dynamic fast scale $v(t)$ with the family of periodic
  solutions $v_{u(t)}(s)$.}
\begin{lemma}\label{lemma:approxv}
  Let $u\in C([0,T])$ be given with
  \begin{equation}\label{la:1}
    u(0)=0,\quad 0\le u'(t)\le C_{A\ref{ass:reaction}a}\epsilon\quad t\in[0,T].
  \end{equation}
  Then, let $v(t)$ be the dynamic solution to~(\ref{motms:v}), i.e.
  \begin{equation}\label{la:2}
    v(0)=v_0,\quad \partial_t v(t) +\lambda\big(u(t)\big) v(t) = f(t)
    \text{ for }t\in [0,T]
  \end{equation}
  and let $v_{u(t)}(s)$ be the family of time-periodic solutions to
  \begin{equation}\label{la:3}
    v_{u(t)}(0)=v_{u(t)}(1)\quad \partial_s v_{u(t)}(s)
    +\lambda\big(u(t)\big) v_{u(t)}(s) = f(s)
    \text{ for }s\in [0,1]\text{ and for all }t\in [0,T]. 
  \end{equation}
  Finally, let $v_0=v_{u(0)}(0)$, i.e. the initial values
  to~(\ref{la:2}) and~(\ref{la:3}) at time $t=0$ agree. Let
  $\lambda(\cdot)$ satisfy 
  Assumption~(\ref{ass:diff}). Then it holds
  \[
  |v(t)-v_{u(t)}(t)| \le C_{L\ref{lemma:approxv}} \epsilon.
  \]
  with a constant $C>0$ that depends on $f$, $\lambda_0$ and on
  Assumptions~\ref{ass:reaction} and~\ref{ass:diff}. 
\end{lemma}
\begin{proof}
  For $v_{u(t)}(t)$ it holds by the chain rule
  \[
  \frac{d}{d_t} v_{u(t)}(t) = v_{u(t)}'(t) + \frac{d
    v_{u(t)}}{du(t)}(t)u'(t), 
  \]
  such that $v_{u(t)}(t)$ is governed by
  \begin{align*}
    \partial_t v_{u(t)}(t) +
    \frac{dv_{u(t)}}{du(t)}(t)u'(t)+ \lambda\big(u(t)\big)
    v_{u(t)}(t) =0 ,\quad v_{u(0)}(0)=v_0. 
  \end{align*}
  Thus, it holds for the difference $w(t):=v(t)-v_{u(t)}(t)$
  \[
  \partial_t w(t) + \lambda\big(u(t)\big) w(t) =
  -\frac{d v_{u(t)}}{du(t)}(t)u'(t),\quad w(0)=0 
  \]
  with the solution
  \begin{equation}\label{diff:solution}
    w(t) = 
    -\int_0^t \frac{d v_{u(s)}}{du(s)}(s)
    u'(s)\exp\Big(-\int_s^t\lambda\big(u(r)\big)\,\text{d}r
    \Big)\,\text{d}s.
  \end{equation}
  To estimate the derivative $\frac{dv_{u(s)}}{du(s)}$ we
  consider two such 
  time-periodic solutions $v_{u(t)}$ and $v_{\eta(t)}$ for fixed $0\le
  u,\eta\le u_{max}$. We estimate their 
  distance by~(\ref{periodic:2}) in Lemma~\ref{lemma:periodic}
  \begin{equation}\label{approxv:11}
    \frac{|v_{u}-v_{\eta}|}{|u-\eta|} \le
    \frac{4}{\lambda_0^2} \|f\|_{L^\infty([0,1])} 
    \Big|\frac{\lambda(u)-\lambda(\eta)}
              {u-\eta}\Big|.
  \end{equation}
  This bound is uniform in $u,\eta$ and $t$, such that
  differentiability of  $\lambda(u)$,~(\ref{mot:3}) gives
  \[
  \Big|\frac{dv_{u(t)}}{du(t)}(t)\Big|=
  \lim_{\eta\to u}\frac{|v_{u(t)}-v_{\eta(t)}|}{|u-\eta|}
  \le \frac{4 C_{A\ref{ass:diff}}}{\lambda_0^2} \|f\|_{L^\infty([0,1])}
  \]
  This allows to estimate~(\ref{diff:solution}) by
  \[
  |w(t)|=|v(t)-v_{u(t)}(t)| \le \frac{4C_{A\ref{ass:diff}} C_{A\ref{ass:reaction}a}}{\lambda_0^2}
  \|f\|_{L^\infty([0,1])}  \epsilon. 
  \]
\end{proof}

%%%%%%%%%%%%%%%%%%%%%%%%%%%%%%%%%%%%%%%%%%%%%%%%%%

{In the previous lemma we investigated the coupling
  from a fixed slow variable $u(t)$  to the fast components $v(t)$ and
  $v_{u(t)}(t)$. This last lemma will study the different evolutions
  of the slow variable $u(t)$ governed
  by~(\ref{motms:kappa}),~(\ref{motms:v}) and of the averaged variable
  $U(t)$ that is determined by equation~(\ref{motms:slow}) with
  periodic micro influences~(\ref{motms:periodic}). }  

\begin{lemma}\label{lemma:conterror}
  Let $(u(t),v(t))$ and $(U(t),v_{U(t)}(t))$ be 
  defined by~(\ref{motms:kappa}), (\ref{motms:v}) 
  and~(\ref{motms:slow}), (\ref{motms:periodic}), respectively, with 
  the initial values $u(0)=U(0)=0$ and $v(0)= v_{U(0)}(0)$. For
  $0\le t\le T=\O{\epsilon^{-1}}$  it holds
  \[
  |U(t)-u(t)| \le C\epsilon,
  \]
  with a constant $C>0$ that depends on the constants from
  Lemma~\ref{lemma:periodic}, \ref{lemma:approxv} and \ref{lemma:conterror}, as well as
  Assumption~\ref{ass:reaction}. 
\end{lemma}
\begin{proof}
  We introduce
  \[
  w(t)\coloneqq  U(t)-\int_{t}^{t+1} u(s)\,\text{d}s,
  \]
  which is governed by
  \[
  w'(t)=\int_t^{t+1}\epsilon\Big( R\big(U(t),v_{U(t)}(s)\big)
  -R\big(u(s),v(s)\big)\Big)\,\text{d}s,\quad
  w(0)=U(0)-\int_0^1u(s)\,\text{d}s \eqqcolon w_0.
  \]
  The initial error is small,  $|w_0| = \O{\epsilon}$,
  compare~(\ref{initialerror}). 
  We insert $\pm R\big(U(t),v(s)\big)$
  \begin{equation}\label{conterror:1}
    w'(t) =
    \int_{t}^{t+1}\epsilon
    \Big(R\big(U(t),v_{U(t)}(s)\big)- R\big(U(t),v(s)\big)\Big)
    + \epsilon\Big(R\big(U(t),v(s)\big) - R\big(u(s),v(s)\big)\Big)
    \,\text{d}s.
  \end{equation}
  Lipschitz continuity of $R(\cdot,\cdot)$,
  Assumption~\ref{ass:reaction}, gives 
  \begin{equation}\label{conterror:2}
    |w'(t)| \le C_{A\ref{ass:reaction}b} \epsilon \int_t^{t+1}
    |v_{U(t)}(s)-v(s)|\,\text{d}s
    +C_{A\ref{ass:reaction}b}\epsilon\int_t^{t+1}
    |U(t)-u(s)|\,\text{d}s.  
  \end{equation}
  The second term is estimated by inserting
  $\pm\int_t^{t+1}u(r)\,\text{d}r$ and by using $|u'|\le C_{A\ref{ass:reaction}a} \epsilon$
  \begin{multline}\label{conterror:3}
    \int_t^{t+1} |U(t)-u(s)|\,\text{d}s
    \le
    \int_t^{t+1}\left|U(t)-\int_t^{t+1}u(r)\,\text{d}r
    \right|\,\text{d}s +
    \int_t^{t+1}\left| \int_t^{t+1}\big(   
    u(r)-u(s)\big)\,\text{d}r\right|\,\text{d}s  \\
    = \int_t^{t+1}|w(t)|\,\text{d}s
    +\int_t^{t+1}\left|\int_t^{t+1}
    \int_s^r u'(x)\,\text{d}x\,\text{d}r\right|\,\text{d}s
    \le |w(t)| + C_{A\ref{ass:reaction}a} \epsilon
  \end{multline}
  To estimate the first term in~(\ref{conterror:2}) we introduce $\pm
  v_{u(s)}(s)$ and use Lemma~\ref{lemma:periodic}, 
  Lemma~\ref{lemma:approxv}, Assumption~\ref{ass:diff}
  (differentiability of $\lambda(u)$) and finally~(\ref{conterror:3})   
  \begin{multline}\label{conterror:4}
    \int_t^{t+1} |v_{U(t)}(s)-v(s)|\,\text{d}s\le
    \int_t^{t+1}\Big(
    |v_{U(t)}(s)-v_{u(s)}(s)| + |v_{u(s)}(s)-v(s)|
    \Big)\,\text{d}s\\
    \le
    C_{L\ref{lemma:periodic}b}C_{A\ref{ass:diff}}\int_t^{t+1}|U(t)-u(s)|\,
    \text{d}s +C_{L\ref{lemma:approxv}} \epsilon  
    \le
    C_{L\ref{lemma:periodic}b}C_{A\ref{ass:diff}}
    \big(|w(t)|+C_{A\ref{ass:reaction}a}\epsilon\big)+
    C_{L\ref{lemma:approxv}} \epsilon   
  \end{multline}
  With
  $C=C(C_{A\ref{ass:reaction}a},C_{A\ref{ass:reaction}b},C_{A\ref{ass:diff}}, 
  C_{L\ref{lemma:periodic}a},C_{L\ref{lemma:periodic}b},C_{L\ref{lemma:approxv}})$ 
  we combine~(\ref{conterror:1})-(\ref{conterror:4}) to find the 
  relation 
  \[
  -C\left( \epsilon + |w(t)|\right)\epsilon \le w'(t) \le
  C\left(\epsilon + |w(t)|\right)\epsilon.
  \]
  An estimate for $|w(t)|$ follows by
  the a bound of the solution  to the corresponding ODE with initial
  value $w(0) = w_0$, where $|w_0|\le C_{\ref{ass:reaction}a}\epsilon$  
  \begin{equation}\label{conterror:5}
   % \left|\bar u(t) - \int_{t-1}^t u(s)\,\text{d}s\right|=
    |w(t)| \le C\epsilon \euler^{C\epsilon t}, 
  \end{equation}
  which satisfies $|w(t)|=\O{\epsilon}$ for $t\le T = \O{\epsilon^{-1}}$.
  Finally, 
  \[
  |U(t)-u(t)| \le 
  |w(t)| +
  \left|\int_t^{t+1} u(s)-u(t)\,\text{d}s\right|
  \le C\epsilon.
  \]
\end{proof}

%

%%%%%%%%%%%%%%%%%%%%%%%%%%%%%%%%%%%%%%%%%%%%%%%%%%
%%%%%%%%%%%%%%%%%%%%%%%%%%%%%%%%%%%%%%%%%%%%%%%%%%
%%%%%%%%%%%%%%%%%%%%%%%%%%%%%%%%%%%%%%%%%%%%%%%%%%
\section{Time discretisation}
\label{sec:timedisc}

{
In this section we introduce second-order time-stepping schemes to
approximate the coupled problem. As in the previous section,
where we derived the multiscale algorithm, we start with the full
plaque growth problem, equation~(\ref{problem:coupled}). Then, the
error analysis for estimating the discretisation error is based on the
simplified model equations~(\ref{problem:simple}).}

The discretisation is based on the second-order Adams-Bashforth scheme
for the slow scale 
and a Crank-Nicolson scheme for the fast scale. Both choices are
exemplarily and can in principle be substituted by any suitable
time-stepping scheme. We choose an explicit scheme for the slow scale
in order to avoid that several fast-scale problems have to be solved
in each time step, see Remark~\ref{rem.impl} below. 

\subsection{Second-order multiscale schemes}

First, we split the time interval $I=[0,T]$ into sub-intervals of
equal size
\begin{equation}\label{KN}
  0=T_0<T_1<\cdots<T_N,\quad K\coloneqq T_n-T_{n-1},
\end{equation}
We define approximations $U_n:=U(T_n)$ based on the second-order
Adams-Bashforth multistep method  
\begin{align}
  \label{MS:AB}
  (AB) && \frac{U_{n+1}-U_n}{K} &=
  \frac{3}{2} \int_0^1\epsilon R(U_n,\vt_{U_n;k})\, \text{d}s
  -\frac{1}{2} \int_0^1\epsilon R(U_{n-1},\vt_{U_{n-1};k}) \, \text{d}s. 
\end{align}
In order to compute the required starting value $U_1$ for the
Adams-Bashforth scheme, we put one forward Euler step at the beginning
of the iteration, which is sufficient to obtain second-order
convergence.

These schemes are formally explicit, they depend however on the
averaged fast scale influences
$R(U_n,\vt_{U_n;k})$
To compute these terms, we introduce a (for simplicity again uniform)
temporal subdivision of the fast periodic interval $I_P=[0,1]$ of step
size $k$
\begin{equation}\label{kn}
  0=t_0<t_1<\cdots<t_{M}=1,\quad k\coloneqq t_m-t_{m-1}. 
\end{equation}
Given a fixed value $0\le U\le u_{max}$, we introduce the notation 
$\vt_{U,m}\coloneqq\vt_{U;k}(t_m)$ and we approximate the periodic
solution on the fast scale with the Crank-Nicolson time-stepping
scheme
\begin{multline}\label{CN:flow}
  \nabla\cdot \vt_{U,m}=0,\quad
  \rho k^{-1}\big(\vt_{U,m}-\vt_{U,m-1}\big)
  +\frac{\rho}{2} \big(
  (\vt_{U,m-1}\cdot\nabla)\vt_{U,m-1}
  +(\vt_{U,m}\cdot\nabla)\vt_{U,m}
  \big)\\
  - \frac{1}{2}\div \big(
  \sigmat(\vt_{U,m-1})+\sigmat(\vt_{U,m})
  \big)=\frac{1}{2}\big( \ft(t_{m-1})+\ft(t_m)\big)
  \quad m=1,\dots,M\\
  \text{such that } 
  |\vt^u_M-\vt^u_0|\le \tolP.
\end{multline}
Based on the approximations made in the previous section, we introduce
the following multiscale method: 

\begin{algo}[Explicit temporal multiscale method]\label{finalalgo}
  Given subdivisions~(\ref{KN}) and~(\ref{kn}) of $I=[0,T]$ and
  $I_P=[0,1]$. Let $U_0=0$. Iterate for $n=1,\dots,N$
  \begin{enumerate}
  \item For $U\coloneqq U_{n-1}$ solve the periodic
    problem~(\ref{CN:flow}) to obtain
    $(\vt_{U_{n-1},m},p_{U_{n-1},m})$ for $m=1,2,\dots,M$. 
  \item Compute the averaged feedback 
    \begin{equation}\label{alg:avg}
      R_{n-1}\coloneqq \frac{k}{M} \sum_{m=1}^M
      R(U_{n-1},\vt_{U_{n-1},m}). 
    \end{equation}
  \item Step forward $U_{n-1} \to U_n$ with the Adams-Bashforth
    method~(\ref{MS:AB}) 
    \begin{equation}\label{alg:AB}
      U_n\coloneqq  U_{n-1}+ \frac{3K}{2}\epsilon R_{n-1} -\frac{K}{2}\epsilon R_{n-2},
    \end{equation}
    or, in the first step, with the forward Euler method
    \begin{equation}\label{alg:FE}
      U_1\coloneqq U_0+K\epsilon R_0. 
    \end{equation}
  \end{enumerate}
\end{algo}

\begin{remark}%[Averaging the feedback]
  In practice we ensure in Step 1 that the solution is periodic up to a certain threshold
  $\|\vt^u_M-\vt^u_0\|\le \tolP$. The box-rule used to
  compute the averaged wall shear stress 
  in Step 2 of the algorithm is
  therefore equivalent to the second order trapezoidal rule (up to the
  small error ${\mathcal O}(k\,\tolP)$).
\end{remark}

The main computational cost comes from the approximation of the periodic
solutions $(\vt_{U,m},p_{u,m})$ for a fixed value of $U$. 
The efficient computation of these periodic problems is discussed below.

%in Section~\ref{sec:periodicproblem}. 

\begin{remark}[Implicit multiscale schemes]\label{rem.impl}
  We are considering the rather simple interaction of the
  Navier-Stokes equations with a scalar ODE. For more detailed
  models, for example a boundary PDE to model the spatially diverse
  accumulation of $u(x,t)$ along the boundary $\Gamma(u)$ or
  even a full PDE/PDE model considering dynamical fluid-structure
  interactions and a detailed modelling of the bio-chemical
  processes causing plaque growth as introduced by Yang, Neuss-Radu et
  al.~\cite{Yang2014,YangJaegerNeussRaduRichter2015}, stiffness issues
  may call for implicit discretisations of the equation for $u$. To
  realise an implicit multiscale method, e.g. based on the
  Crank-Nicolson scheme for both temporal scales an outer iteration must
  be introduced. We refer to~\cite{MizerskiRichter2020} for a first
  application of the multiscale scheme to a PDE/PDE coupled flow
  problem. 
\end{remark}

\subsection{Error analysis for the model problem}

We consider the system of equations \eqref{motms:kappa}-\eqref{motms:v}, its 
multiscale approximation \eqref{motms:slow}-\eqref{motms:periodic} 
and the discrete problem \eqref{CN:flow}-\eqref{alg:FE}. Concerning
the short-scale problem, we make the following assumption.

\begin{assumption}[Approximation of the periodic problem]
  \label{ass:approxcn}
  Let $\tolP>0$ be the tolerance for reaching periodicity.
  We assume that there exists a constant $C>0$ such that the
  Crank-Nicolson discretisation~(\ref{CN:flow}) to the flow problem
  satisfies the bound
  \[
  \|\vt_{U,M}-\vt_{U,0}\| +
  \max_{m=1,\dots,M}\|\vt_{U}(t_m)-\vt_{U,m}\|\le
  C k^2  + \tolP
  \]
  where  the constant $C$ in particular does not depend on $\epsilon$ and $U$.  
\end{assumption}

\begin{remark}[Approximation of the periodic problem]\label{rem.approxcn}
  Considering ODEs, the 
  error estimate for the trapezoidal scheme is standard and can
  be found in many textbooks. Applied to the Navier-Stokes equations
  optimal order estimates under realistic regularity assumptions are
  given in~\cite{HeywoodRannacher1990}. Similar estimates that also
  include second-order in time estimates for the pressure (which
  might be required for a stress-based feedback) are given
  in~\cite{SonnerRichter2019}.  
  For algorithms to control the periodicity error, we refer to Section~\ref{sec:periodic} below.
\end{remark}

\begin{lemma}[Regularity of the solution]\label{lemma:regularity}
  Let $U(t)$ be the solution to \eqref{motms:slow} for $v_U\in C(0,T)$
  and $0\leq U\leq u_{\max}$. Moreover, let the map $U \mapsto
  \lambda(U)$  be twice differentiable with bounded second
  derivatives. It holds
  \[
  U\in C^3(I),\quad \max_{[0,T]}|U''| =
  \O{\epsilon^2},\quad \max_{[0,T]}|U'''| =
  \O{\epsilon^3}
  \]
\end{lemma}
\begin{proof}
Let us first note, that $U'$ is bounded due to the continuity of the
right-hand side $R\big(U(t),v_{U(t)}(s)\big)$ of \eqref{motms:slow}. 
Next, we consider the (total) temporal derivative of the right-hand
side. The chain rule gives 
\[
\begin{aligned}
  d_t \int_0^1 R\big(U(t),v_{U(t)}(s)\big) \,\text{d}s 
  = R\big(U(t),v_{U(t)}(1)\big) - R\big(U(t),v_{U(t)}(0)\big)
  + \int_0^1 d_t R\big(U(t),v_{U(t)}(s)\big) \,\text{d}s. 
\end{aligned}
\]
The first part vanishes due to the periodicity of $v_{U(t)}$.
For the second part we have with~(\ref{reaction:simple})
\[
d_t R\big(U(t),v_{U(t)}(s)\big)=
-\frac{ U'(t)}{1+U(t)^2} \int_0^1 \frac{1}{ 1+(v_{U(t)}(s))^2} \,\text{d}s
- \frac{1}{1+U(t)} \int_0^1 \frac{2
  v_{U(t)}(s) d_t v_{U(t)}(s)}{
  \left(1+(v_{U(t)}(s))^2\right)^2} \,\text{d}s 
\]
As in the proof of Lemma~\ref{lemma:approxv} we show
\begin{align*}
  \Big|d_t v_{U(t)}\Big| = \Big|\frac{d v_{U(t)}}{dU} U'\Big| ={\cal
    O}(\epsilon). 
\end{align*}
In combination with the bound $|U'(t)|\leq c\epsilon$, we obtain
\begin{align*}
  |U''(t)| = \Big| d_t \int_0^1\epsilon
  R\big(U(t),v_{U(t)}(s)\big) \,\text{d}s  \Big| \leq
  C\epsilon^2. 
\end{align*}
A similar argumentation yields for the third derivative
\begin{align*}
  |U'''(t)| = \Big| d_t^2 \int_0^1 \epsilon R\big(U(t),v_{U(t)}(s)\big)
  \,\text{d}s  \Big| = \Big| \int_0^1 d_t^2\epsilon R\big(U(t),v_{U(t)}(s)\big)
  \,\text{d}s  \Big| \leq C\epsilon^3, 
\end{align*}
where we have used that
\begin{align*}
  d_t^2 v_{U(t)} = \frac{d v_{U(t)}}{dU}  U''  +\frac{d^2v_{U(t)}}{d
    U^2}  U'^2  = {\cal O}(\epsilon^2), 
\end{align*}  
given that $\lambda(U)$ is twice differentiable in $U$.
\end{proof}

\begin{lemma}[Local approximation error of the effective equation]
  \label{lemma:Locdiscerror}
  Let $U\in C^3(I)$ be the solution
  to~(\ref{motms:slow})-(\ref{motms:periodic}), 
  $U_{K;k}\in \mathds{R}^{N+1}$ the approximation given by
  Algorithm~\ref{finalalgo}.
  For $T_n={\cal O}(\epsilon^{-1})$, the error $E_n := U(T_n) - U_{n;k}$
  is bounded by 
  \[
  |E_n| \le C\big( \epsilon^2 K^2 + k^2 +
  \tolP\big) 
  \]
  with a constant $C>0$ that does not depend on $\epsilon$, $K$, $k$
  and $\tolP$. 
\end{lemma}
\begin{proof}
 Combination of Taylor expansions around $T_n$ and $T_{n-1}$ 
 of the continuous solution $U$ gives
 \begin{multline*}
   U (T_{n+1}) 
   = U (T_n)
   + \frac{3K}{2} \int_0^1\epsilon R\big(U(T_n),v_{U(T_n)}(s)\big) \,\text{d}s\\
   - \frac{K}{2}  \int_0^1
   \epsilon R\big(U(T_{n-1}),v_{U(T_{n-1})}(s)\big)\,\text{d}s 
   + {\cal O}(K^3) \max_{\xi\in (T_{n-1}, T_{n+1})} |U'''(\xi)|,
 \end{multline*}
 where $\xi\in [T_{n-1},T_n]$. 
 In combination with (\ref{MS:AB}), we obtain the error representation
 \begin{multline*}
   E_{n+1} = E_n
   + \frac{3K}{2}  \int_0^1\epsilon
   \Big(R\big(U(T_n),v_{U(T_n)}(s)\big)-
   R\big(U_n,v_{U_n;k}(s)\big)\Big) \,\text{d}s\\ 
   - \frac{K}{2} \int_0^1 \epsilon\Big( R\big(U(T_{n-1}),v_{U(T_{n-1})}(s)\big)
   - R\big(U_{n-1},v_{U_{n-1},k}(s)\big)\Big) \,\text{d}s 
   + {\cal O}(K^3) \max_{\xi\in (T_{n-1}, T_{n+1})} |U'''(\xi)|
 \end{multline*}
 With the Lipschitz continuity of $R(\cdot)$,
 Assumption~\ref{ass:reaction} we 
 estimate 
 \begin{equation*}
   \int_0^1 \Big| R\big(U(T_n),v_{U(T_n)}(s)\big)-
   R\big(U_n,v_{U_n;k}(s)\big) \Big|\, \text{d}s
   \leq C\left(\int_0^1 \big| v_{U(T_n)}(s) - 
   v_{U_n;k}(s)\big|\,\text{d}s + \big| U(T_n) - U_n\big| \right). 
 \end{equation*}
 For the first part, we use the estimate (\ref{periodic:2}) of
 Lemma~\ref{lemma:periodic} and Assumption~\ref{ass:approxcn}
 \begin{align*}
   \int_0^1 \big| v_{U(T_n)}(s) - v_{U_n;k}(s)\big| \,\text{d}s &\leq
   \int_0^1 \big|v_{U(T_n)} - v_{U_n}\big|\,\text{d}s  
   + \int_0^1 \big| v_{U_n} - v_{U_n;k}\big|\,\text{d}s \\
   &\leq
   C_{L\ref{lemma:periodic}b}C_{A\ref{ass:approxcn}}
   \left(|U(T_n)-U_n| + k^2\|v_{U_n}\|_{C^3([0,1])}+\tolP \right)
 \end{align*} 
 In combination with Lemma~\ref{lemma:regularity} this yields
 \begin{align*}
   |E_{n+1}| \leq |E_n| +
   C
   \epsilon K  \left( | E_n| + |E_{n-1}| 
   + k^2 + \tolP  + \epsilon^2 K^2\right). 
 \end{align*}
 with
 $C=C(C_{L\ref{lemma:periodic}b},C_{A\ref{ass:approxcn}},C_{\ref{lemma:regularity}})$. 
 Summing over $n=1,\dots,N-1$ and using $E_0=0$, we obtain 
 \[
 |E_N| \leq |E_1| + C
 \epsilon 
 T_N \big(k^2  + \tolP + \epsilon^2 K^2\big)
 +C
 \sum_{n=1}^{N-1} \epsilon K |E_n|. 
 \]
 The term $|E_1|$ depends on the forward Euler method, that is used to
 compute $U_{K,1}$
 \begin{align*}
   |E_1| \leq CK^2 |U''|_{\infty} \leq C\epsilon^2 K^2,
\end{align*}
 where we have used Lemma~\ref{lemma:regularity} and $|E_0|=0$.
 Finally, a discrete Gronwall inequality yields
 \[
 |E_N| \leq C \epsilon T_N 
 \Big( k^2 + \epsilon^2 K^2 + \tolP
 \Big) 
 \exp\Big(\epsilon T_N\Big). 
 \]
 The postulated result follows for $T_N=\O{\epsilon^{-1}}$.\\
\end{proof}

Finally, we can estimate the error between the multiscale algorithm
and the solution $u(t)$ to the original coupled problem.
\begin{theorem}[A priori estimate for the multiscale algorithm]
  \label{thm:main}
  Let $I=[0,T]$ with $T=\O{\epsilon^{-1}}$ and let $u\in C(I)$ and $U_K$ be 
  the solutions to the original problem~(\ref{problem:coupled}) and
  the discrete effective equations (\ref{MS:AB}), respectively. It
  holds
  \[
  |u(T_n)-U_n| = C \Big(k^2 + \epsilon^2 K^2 + \tolP +
  \epsilon\Big),
  \]
  where $C>0$ does
  not depend on $\epsilon,K,k$ and $\tolP$. 
\end{theorem}
\begin{proof}
  We introduce $\pm U(T_n)$ and estimate
  \[
  |u(T_n)-U_n| \le 
  |u(T_n)-U(T_n)| +
  |U(T_n)-U_n|.
  \]
  The two terms on the right hand side are estimated with
  Lemma~\ref{lemma:conterror} 
  and Lemma~\ref{lemma:Locdiscerror}.
\end{proof}

\subsection{Approximation of the periodic flow problem}
\label{sec:periodic}

The temporal multiscale schemes are based on periodic solutions
$\vt_{U_n}(s)$ for $s\in [0,1]$, where the variable
$U_n=U(t_n)$ is fixed such that no feedback between fluid problem and
reaction equation takes place within this short interval. A numerical difficulty lies in 
the determination of the correct initial value $
\vt_{U_n,0}$ that yields periodicity
$\vt_{U_n}(0) = \vt_{U_n}(1)$. Let us consider again the full
Navier-Stokes problem 
\begin{equation}\label{cycle}
  \nabla\cdot \vt_{U_n}=0,\quad 
  \rho(\partial_t \vt_{U_n} + (\vt_{U_n}\cdot\nabla)\vt_{U_n})
  -\div\,\sigmat(\vt_{U_n},p_{U_n}) = \ft,\quad
  \vt_{U_n}(1) = \vt_{U_n}(0) \text{ in }\Omega(U_n).
\end{equation}
We assume that such a periodic solution of
the  
Navier-Stokes equations exists. Some results are given by Kyed and
Galdi~\cite{GaldiKyed2016a,GaldiKyed2016,Kyed2012} that
require, however, severe restrictions on the problem 
data.
Depending on the transient dynamics, the decay of the nonstationary
solution to this periodic solution can be very slow. It 
depends basically on $\exp(-\nu\lambda_0)$, where $\nu$ is the
viscosity and $\lambda_0>0$ the smallest eigenvalue of the Stokes
operator, which depends on the inverse of the domain size.
Several acceleration techniques that are based 
on shooting
methods~\cite{GovindjeePotterWilkening2014,ZahrPerssonWilkening2016},
Newton 
schemes~\cite{Steuerwalt1979,JianBieglerFox2003} or 
gradient-based 
optimisation
techniques~\cite{HanteMommerPotschka2015,RichterWollner2017} have been
proposed 
to quickly identify 
the initial value $\vt_{U_n,0}$.
{We apply a simple acceleration scheme that is based
  on decomposing the periodic solution into its average and the
  fluctuations, see~\cite{Richter2020}. Here, we shortly recapitulate
  the algorithm that has been 
  introduced in~\cite{Richter2020}.
  \begin{algo}[Averaging scheme for the identification of periodic
      solutions]
    Given the initial value $\vt_{U_n,0}^{(1)}$, usually
    $\vt_{U_n,0}^{(1)}\coloneqq \vt_{U_{n-1},0}$ and let $\tolP>0$ be
    a given tolerance. Iterate for
    $l=1,2,\dots$
    \begin{enumerate}
    \item Solve one cycle of~(\ref{cycle}) for
      $(\vt_{U_n}^{(l)},p_{U_n}^{(l)})$ with the initial
      $\vt_{U_n}^{(l)}(0)=\vt_{U_n,0}^{(l)}$ 
    \item Compute the error in periodicity 
      \[
      \errP^{(l)}\coloneqq \|\vt_{U_n}^{(l)}(1)-\vt_{U_n}^{(l)}(0)\|
      \]
    \item Stop, if $\errP^{(l)}<\tolP$. 
    \item Compute the velocity average over the cycle
      \[
      \bar \vt_{U_n}^{(l)}\coloneqq
      \int_0^1\vt_{U_n}^{(l)}(s)\,\text{d}s 
      \]
    \item Compute the stationary update problem for
      $(\bar\wt_{U_n}^{(l)},\bar q_{U_n}^{(l)})$
      \begin{equation}\label{update}
        \nabla\cdot \bar\wt^{(l)}_{U_n}=0,\quad 
        \rho\big(
        (\bar\vt_{U_n}^{(l)}\cdot\nabla)\bar\wt_{U_n}^{(l)} +
        (\bar\wt_{U_n}^{(l)}\cdot\nabla)\bar\vt_{U_n}^{(l)}
        \big)
        -\div\,\sigmat(\bar\wt_{U_n}^{(l)},\bar q^{(l)}_{U_n}) =
        \vt_{U_n}^{(l)}(1)-\vt_{U_n}^{(l)}(0)
      \end{equation}
    \item Update the initial
      \[
      \vt^{(l+1)}_{U_n,0}\coloneqq \vt^{(l)}_{U_n}(1) + \bar\wt^{(l)}_{U_n}
      \]
      \sfrei{and go to step 1}.
    \end{enumerate}
  \end{algo}
  The basic idea of introducing the averaged update
  problem~(\ref{update}) in Step 5 of the algorithm is to quickly
  predict the correct average of the periodic solution. The
  computational effort for each iteration lies mostly in \sfrei{Step 1}, where
  one complete nonstationary cycle of the periodic problem over the
  period $[0,1]$ is computed. Given the step size $k$ this means
  solving $k^{-1}$ time steps of the discrete Navier-Stokes
  problem. In addition, Step 5 calls for the solution of one
  additional stationary problem.}

{Using this scheme we are able
  to reduce the periodicity error to $\tolP<10^{-4}$ in less than 5
  cycles of the algorithm. In the context of usual HMM approaches this
  would correspond to choosing the relaxation time as
  $\eta=\unit[5]{s}$ in terms of computational effort, i.e. 5 times the period length, see~\cite{Abdulleetal2012}. 
  }

%%%%%%%%%%%%%%%%%%%%%%%%%%%%%%%%%%%%%%%%%%%%%%%%%%
%%%%%%%%%%%%%%%%%%%%%%%%%%%%%%%%%%%%%%%%%%%%%%%%%%
%%%%%%%%%%%%%%%%%%%%%%%%%%%%%%%%%%%%%%%%%%%%%%%%%%

%%%%%%%%%%%%%%%%%%%%%%%%%%%%%%%%%%%%%%%%%%%%%%%%%%
%%%%%%%%%%%%%%%%%%%%%%%%%%%%%%%%%%%%%%%%%%%%%%%%%%
%%%%%%%%%%%%%%%%%%%%%%%%%%%%%%%%%%%%%%%%%%%%%%%%%%

\section{Numerical examples}
\label{sec:num}

We consider the full problem described in the introduction, namely the
incompressible Navier-Stokes equations coupled to a scalar ODE
model. In order to transfer the proofs from the simplified setting to this
more relevant case, several open questions regarding the existence and
regularity theory of the Navier-Stokes equations in the periodic
setting would have to be addressed. The numerical results presented in
this section will, however, reveal convergence rates and error
constants that are in full agreement with the theoretical findings for
the simplified model problem.

\subsection{Configuration of the plaque formation problem}
\label{sec:config:plaque}

A sketch of the plaque growth problem is given in
Figure~\ref{fig:config}, the governing equations have been outlined in
the introduction, Section~\ref{sec:intro}.
On the fast scale we consider a  Navier-Stokes flow in a channel whose
width depends on the \emph{slowly evolving variable} $u(t)$. The
variable domain describing the channel is given by 
\begin{equation}\label{domain:2}
  \Omega(u) = \left\{ (x,y)\in\mathds{R}^2\,:\,
  -\unit[5]{cm}<x<\unit[10]{cm},\; 
  |y|<\unit[(1.5-u \gamma(x))]{cm}\right\},\quad
  \gamma(x)=\exp\left(-x^2\right).
\end{equation}
Instead of a complex growth model for the plaque formation as
introduced in~\cite{YangJaegerNeussRaduRichter2015} we use this
explicit dependence of the domain on the scalar $u(t)$. The periodic
Navier-Stokes problem is driven by a time periodic Dirichlet condition
on the inflow boundary $\Gamma_{in}$ 
\begin{equation}\label{inflowprofile}
  \vt_{in}(y,t) = 25\sin(\pi t)^2
  \left(1-\frac{y^2}{1.5^2}\right)
  \unit{cm/s} \quad \text{ on }  \Gamma_{in}\times [0,T].
\end{equation}
%
%where $r=y$ in 2 dimensions and $r=\sqrt{y^2+z^2}$ in 3
%dimensions. 
%
On the outflow boundary $\Gamma_{out}$ we
specify the \emph{do-nothing} outflow condition
\begin{equation}\label{outflow}
  \rho\nu \partial_{\nt}\vt - p\nt = 0
\end{equation}
that includes a pressure normalising condition
$\int_{\Gamma_{out}}p\,\text{d}s=0$, 
see~\cite{HeywoodRannacherTurek1992}. Kinematic viscosity and density
resemble blood and the parameters in the reaction
term~\eqref{reaction} are tuned to obtain a realistic behaviour concerning the different temporal
scales of atherosclerotic plaque growth
\begin{equation}\label{reaction2}
  \rho = \unit[1]{g/cm^3},\quad \nu=\unit[0.04]{cm^2\cdot
    s^{-1}}, \quad \sigma_0 = 30.
\end{equation}
The constant $\sigma_0$ is such that 
the concentration $u$ reaches the value $1$ at
approximately $T={\mathcal O}(\epsilon^{-1})$.

We exploit the symmetry of the problem and compute on the upper half
of the domain only. On
the symmetry boundary at $y=0$ we prescribe the condition
\[
\vt\cdot\nt = 0,\quad \sigmat(\vt,p)\nt\cdot\taut =0.
\]

Problem \eqref{problem:coupled} can be formulated on a reference
domain $\Omega\coloneqq\Omega(0)$ by means of an \emph{Arbitrary Lagrangian
  Eulerian} approach, see~\cite{Donea1982} or \cite[Chapter
  5]{Richter2017}, using the map
\begin{equation}\label{ALE:map}
  T:\Omega(0)\to \Omega(u),\quad
  T(u(t);x,y) = \begin{pmatrix}
    x \\
    \frac{1.5-u(t)\gamma(x)}{1.5} y
  \end{pmatrix},\quad \gamma(x)=\exp(-x^2),
\end{equation}
with derivative and determinant given by 
\begin{equation}
  \begin{aligned}  
    \quad \Ft&:=\hat\nabla T =
    \begin{pmatrix}
      1& 0 \\
      -\frac{u\gamma'(x)}{1.5}y &\frac{1.5-u\gamma(x)}{1.5}
    \end{pmatrix},\quad J:=\det(\Ft) = \frac{1.5-u\gamma(x)}{1.5}. 
  \end{aligned}
\end{equation}
The Navier-Stokes equations mapped to the reference domain take the
form
\begin{equation}\label{ALE:NS}
  \begin{aligned}
    \div\left( J\Ft^{-1}\vt\right) = 0,\quad 
    \rho J\left( \partial_t\vt +  (\Ft^{-1}(\vt-\partial_tT)\cdot\nabla)
    \vt\right)  -
    \div\left( J\Ft^{-T}\hat\sigmat(\vt,p)\Ft^{-1}\right) 
    &=0
    \quad \text{ in }\Omega\\
    \vt(0)=\vt_0,\qquad
    \hat\sigmat(\vt,p)\coloneqq -pI + \rho_f\nu_f \nabla\vt\Ft^{-1}. 
  \end{aligned}
\end{equation}
Formulations~(\ref{ALE:NS}) and~(\ref{problem:coupled}) are equivalent
as long as $J>0$, which we can guarantee if we limit the maximum
deformation by $u_{max}\coloneqq 1$. 
{The resulting Reynolds number 
  \[
  Re=\frac{\bar \vt L}{\nu} 
  \]
  with the channel diameter  $L=\unit[3]{cm}$, the kinematic
  viscosity $\nu=\unit[0.04]{cm^2\cdot s^{-1}}$ and the
  flow rate $\bar v=(3-2U)^{-1}\unit{cm\cdot s^{-1}}$ in the remaining
  gap of width $3-2U(t)$ is in the range of $0$ and about $3\,750$ as
  long as $U\le u_{max}=1$. Such high values are only reached at peak
  inflow, compare~(\ref{inflowprofile}). 
}

The correct model and a full comprehension of shear
effects on plaque formation and growth are still under active
discussion. It is however understood that regions of (relatively) low
shear stress that exhibit an oscillatory character are more prone to
plaque growth~\cite{Whale2006,Dhawan2010}. The reaction
term (\ref{reaction}) mimics this behaviour. Its
dependence on the flow problem by means of the wall shear stress is
nonlinear and cannot be considered by a simple averaging as done
in~\cite{Yang2014,YangJaegerNeussRaduRichter2015,
  YangRichterJaegerNeussRadu2015}.

%%%%%%%%%%%%%%%%%%%%%%%%%%%%%%%%%%%%%%%%%%%%%%%%%%

{
  \subsection{Significance of the model problem and application to the
    plaque formation model}
  \label{sec:relevance}
  The analysis of the temporal multiscale scheme was based on
  Assumptions~\ref{ass:reaction},~\ref{ass:periodic},~\ref{ass:diff} and
  and~\ref{ass:approxcn}. Further, we have used the linearity
  of the model problem and the availability of analytical solutions to
  the ODEs appearing in the model problem. Here, we will shortly
  motivate the relevance of the simplified model problem and discuss the
  application of the multiscale scheme to the full plaque formation
  model.}
  
{If we linearise the Navier-Stokes equations~(\ref{ALE:NS}) by omitting
the convective terms $((\vt-\partial_t T)\cdot\nabla)\vt$, we obtain the Stokes
equations on $\Omega(u)$. These have a system of $L^2$-orthonormal
eigenfunctions with eigenvalues $0<\lambda_0(u)\le
\lambda_1(u)\le\cdots$. As long as the domain does not deteriorate,
i.e. for $0\le u\le u_{max}$ it holds $\lambda_0(u)\ge \lambda_0>0$. 
Due to the regularity of the reference map \eqref{ALE:map}, the
mapping of the equations to the reference domain is also
differentiable. Its derivative \eqref{mot:3} is bounded
as long as contact of the boundary walls is prevented, i.e. as long as
$u\le u_{max}$ is bound away from 1.5, compare \eqref{domain:2} and
\eqref{ALE:map} (Assumption~\ref{ass:diff}). By diagonalisation of the Stokes problem with respect
to the system of eigenfunctions we reduce the problem to a system of
ordinary differential equations of type \eqref{motms:v}.
Lemma~\ref{lemma:periodic} can be applied to each component of the
diagonalized system. However, since the eigenvalues of the Stokes
operator are not bounded a formal extension to the full Stokes problem
requires further steps.}

{The essential assumptions for the application of the multiscale method
is the boundedness and the Lipschitz continuity of the reaction term
$R(\cdot,\cdot)$ with respect to slow and fast variables (Assumption~\ref{ass:reaction}) as well as
the existence of time-periodic solutions to the isolated fast scale
problem (Assumption~\ref{ass:periodic}). Given a fixed value of $u$, the fast scale
problem~(\ref{periodic}) is given by the Navier-Stokes equations on
the domain $\Omega(u)$. 
The unique existence of periodic solutions to the Navier-Stokes
equations is only guaranteed for small problem data,
see~\cite{GaldiKyed2016a,GaldiKyed2016,Kyed2012}. These results will
most likely not apply to the higher Reynolds number regime of typical
blood flow configurations, such that Assumption~\ref{ass:periodic}
can not be verified in our setting. However, given a periodic solution,
since $\ft=0$, the Dirichlet data $\vt_{in}$ is smooth and since the
domain allows for a piecewise $C^\infty$ parametrisation with a finite
number of convex corners, we expect the regularity
\begin{equation}\label{ass:nsreg}
  \sup\big(\|\vt(t)\|_{H^2(\Omega)} + \|p(t)\|_{H^1(\Omega)}\big) \le
  C_{\ref{ass:nsreg}}, 
\end{equation}
see~\cite{HeywoodRannacher1982}. Under this assumption we can show
Lipschitz continuity and boundedness of the reaction term.
\begin{lemma}[Lipschitz continuity]
  \label{lemma:lip}
  Let $u\in C(I)$ with $0\le u(t)\le u_{max}$. Assume that for fixed
  $u$, the time-periodic Navier-Stokes problem allows for a unique
  solution $\vt_u$ satisfying~(\ref{ass:nsreg}) with 
  $C_{\ref{ass:nsreg}}=C(u_{max})$. Then, the reaction
  term~(\ref{reaction}) is bounded
  \begin{equation}\label{ll:1}
    |R(u,\vt_u)| \le 1,
  \end{equation}
  and Lipschitz continuous with respect to both arguments
  \begin{subequations}
    \begin{align}
      \label{L1}
      |R(u,\vt)-R(\eta,\vt)| &\le |u-\eta|
      &\forall u,\eta&\in [0,u_{max}],\quad
      \forall \vt\in H^2(\Omega),\\
      \label{L3}
      |R(u,\vt)-R(u,\ut)| &\le L |\ut-\vt|
      &\forall \vt,\ut&\in H^2(\Omega),
      \quad \forall u\in [0,u_{max}],
    \end{align}
  \end{subequations}
  with a constant $L>0$. 
\end{lemma}
\begin{proof}
  Given~(\ref{ass:nsreg}), the wall shear stress is well defined
  \begin{equation}\label{est:wss}
    |\sigma_{WSS}(\vt)|\le 2\rho\nu\sigma_0^{-1} c_{tr} \|\vt\|_{H^2(\Omega)},
  \end{equation}
  where $c_{tr}$ is the constant of the trace inequality
  $\|\nabla \vt\|_{\Gamma}\le c_{tr} \|\vt\|_{H^2(\Omega)}$. 
  Then~(\ref{ll:1})
  directly follows by the construction of $R(\cdot,\cdot)$,
  see~(\ref{reaction}).
  Further, it holds
  \[
  \big|R(u,\vt)-R(\eta,\vt)\big| =
  \big(1+|\sigma_{WSS}(\vt)|^2\big)^{-1}
  \big(1+u\big)^{-1}\big(1+\eta\big)^{-1}|u-\eta|,
  \]
  which shows~(\ref{L1}). Likewise
  \[
  \big|R(u,\vt)-R(u,\ut)\big| =
  \big(1+u\big)^{-1}
  \frac{\big|\sigma_{WSS}(\vt)+\sigma_{WSS}(\ut)\big|}
  {\big(1+|\sigma_{WSS}(\vt)|^2\big)
    \big(1+|\sigma_{WSS}(\ut)|^2\big)}
  \big|\sigma_{WSS}(\vt)-\sigma_{WSS}(\ut)\big|
  \]
  Since the wall shear stress is a linear functional
  $\sigma_{WSS}(\vt)+\sigma_{WSS}(\ut)=\sigma_{WSS}(\vt+\ut)$ 
  and due to the relation $2\sigma_{WSS}(\vt) \leq 1+\sigma_{WSS}(\vt)^2$,
  we estimate
  \[
  \big|R(u,\vt)-R(u,\ut)\big|\le
  %2\rho\nu \sigma_0^{-1}c_{tr}
  \|\vt-\ut\|_{H^2(\Omega)},
  \]
  see~(\ref{est:wss}).
\end{proof}
Finally, the validity of Assumption~\ref{ass:approxcn} has been discussed in Remark~\ref{rem.approxcn}.}

\subsection{Discretisation}

We briefly sketch the discretisation in space and time. All numerical
experiments have been realised in the software library Gascoigne
3D~\cite{Gascoigne3D}. 
% In time, we introduce discrete steps $t_n=k\cdot n$ based on a uniform
% step size $k>0$. We denote by $\vt_n = \vt(t_n)$ and $c_n = u(t_n)$
% the approximation at 
% time $t_n$ and discretise with the Crank-Nicolson scheme
% %
% \begin{multline}\label{ALE:timestepping}
%   \div\left( J^n\Ft_n^{-1}\vt^n\right) = 0,\quad
%   \rho \bar J^n ( \vt^n-\vt^{n-1})
%   + \frac{k}{2} A(c^n;\vt^n,p^n)
%   + \frac{k}{2} A(c^{n-1};\vt^{n-1},p^{n})
%   = 0\\
%   c^n-c^{n-1} = k \epsilon \left( \theta \frac{\epsilon r[\vt^n]}{1+
%     c^n} + (1-\theta) \frac{\epsilon r[\vt^{n-1}]}{1+
%      c^{n-1}}\right).
% \end{multline}
% %
% Note that the divergence condition and the pressure are taken fully
% implicit. For better stability an implicitely shifted variant of the
% Crank-Nicolson scheme is possible. 
We use uniform time-steps $k$ and $K$ on both scales and the time-stepping schemes 
presented in \eqref{MS:AB}-\eqref{CN:flow}.

For spatial discretisation we triangulate the reference domain
$\hat\Omega$ into open quadrilaterals, allowing for local refinement
based on hanging nodes, see~\cite{Richter2012a} for details on the
realisation in the software Gascoigne 3d. Equal-order biquadratic
finite elements are used for velocity and pressure degrees of
freedom. Pressure stabilisation is accomplished with
the local projection stabilisation
scheme~\cite{BeckerBraack2001}. Stabilisation of the convective terms
is not required due to the moderate Reynolds numbers.

\paragraph{Direct simulation}

The PDE/ODE system is a multiscale coupled problem. We will compare the presented 
multiscale scheme with a direct forward simulation. As we do not
expect any stiffness-related problems in the ODE we decouple the PDE/ODE system by an
implicit/explicit approach where, as in the multiscale approach, the discretisation of the
Navier-Stokes equation~(\ref{ALE:NS}) is based on the second-order Crank-Nicolson
scheme and the discretisation of the ODE on the second-order explicit
Adams-Bashforth formula
%   \[
%   \begin{aligned}
%     U_{n+1} -  U_n&=  \frac{3k}{2} \epsilon R(U_n,\vt_n) - 
%     \frac{k}{2} \epsilon R(U_{n-1},\vt_{n-1}),\\
%     \rho  J_{n+1} (\vt_{n+1}-\vt_{n}) &=  \frac{k}{2}
%     A(\bar u_{n+1};\vt_{n+1},p_{n+1}) +\frac{k}{2}
%     A(\bar u_n;\vt_{n},p_{n+1}),\quad
%     B(\bar u_{n+1},\vt_{n+1})=0,
%   \end{aligned}
%   \]
resulting in a multiscale method that splits naturally into one
explicit ODE-step and an implicit Navier-Stokes step. 

% \subsubsection{Discretisation of the averaged problem}
% 
% We fix $\bar c$ such that $\Ft=\Ft(\bar c)$ and $J=J(\bar c)$ do not
% depend on the time step 
% and such that $\partial_t T=0$. Summing (\ref{CN:flow})
% from $n=1$ to $n=N$ shows the following relation for a discrete
% solution
% %
% \begin{multline*}
%   k  \div\left( J\Ft^{-1} \sum_{n=1}^N \vt^n \right) =0,\quad
%   k\rho J\Ft^{-1}\left(
%   (1-\theta)\vt^0\cdot\nabla\vt^0 
%   +\sum_{n=1}^{N-1} \vt_n\cdot\nabla\vt_n + \theta \vt_N\cdot\nabla\vt_N
%   \right)\\
%   - k\rho_f\nu_f \left(
%   \div\left(J\Ft^{-T} 
%   \left((1-\theta)\nabla\vt_0 + 
%   \sum_{n=1}^N\nabla\vt_n +
%   \theta\nabla\vt_N\right)\Ft^{-1}\Ft^{-1}\right)
%   \right) + k J\Ft^{-T}\Ft^{-1} \sum_{n=1}^N p_n\\
%   =\rho J(\vt_0-\vt_N) 
% \end{multline*}

\subsection{Numerical analysis of the plaque formation problem}

\subsubsection{Configuration with resolvable time scales}

\begin{figure}[t]
  \begin{center}
    \includegraphics[height=0.25\textwidth]{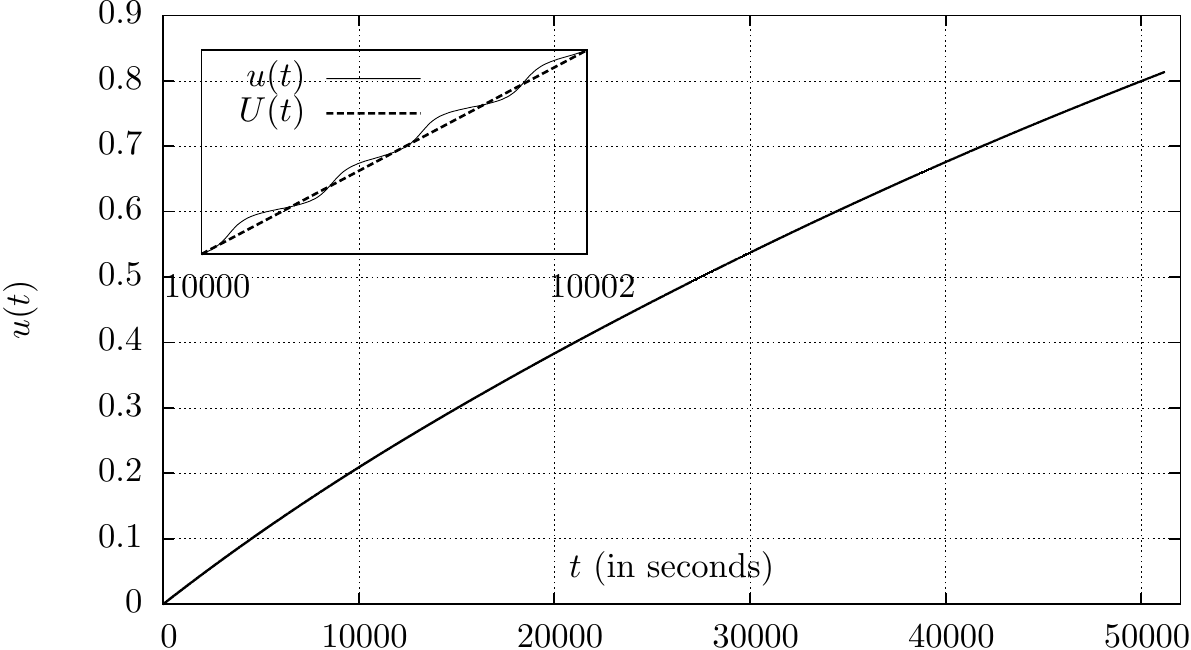}
    \hspace{0.02\textwidth}
    \includegraphics[height=0.25\textwidth]{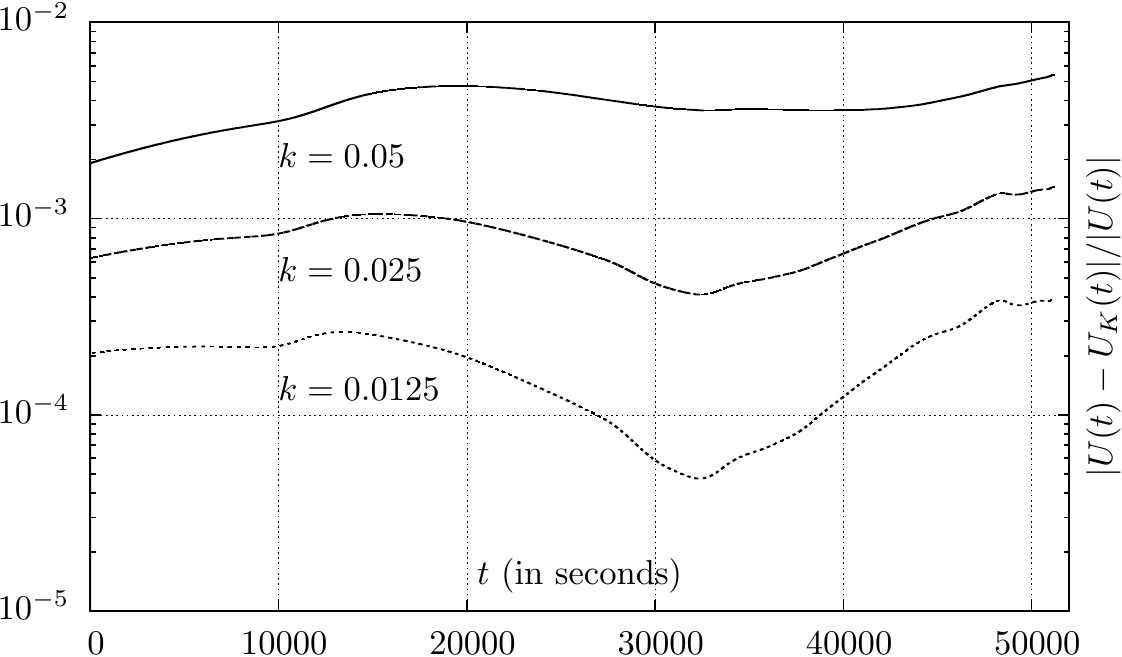}
  \end{center}
  \caption{\emph{Left}: Evolution of the concentration variable $U(t)$
    as function over time (forward simulation with $k=0.05$s). In the
    small subplot we show both $U(t)$ and the resolved
    variable $u(t)$. The maximum deviation is bound by
    $\max|u(t)-U(t)|\le 3\cdot  10^{6}$. 
    \emph{Right}:
    Relative error in $U(t)$ under refinement of the time step
    $k$ (compared to extrapolated values).
  \label{fig:num1:1}}
\end{figure}

\begin{figure}[t]
  \begin{center}
    \includegraphics[width=0.48\textwidth]{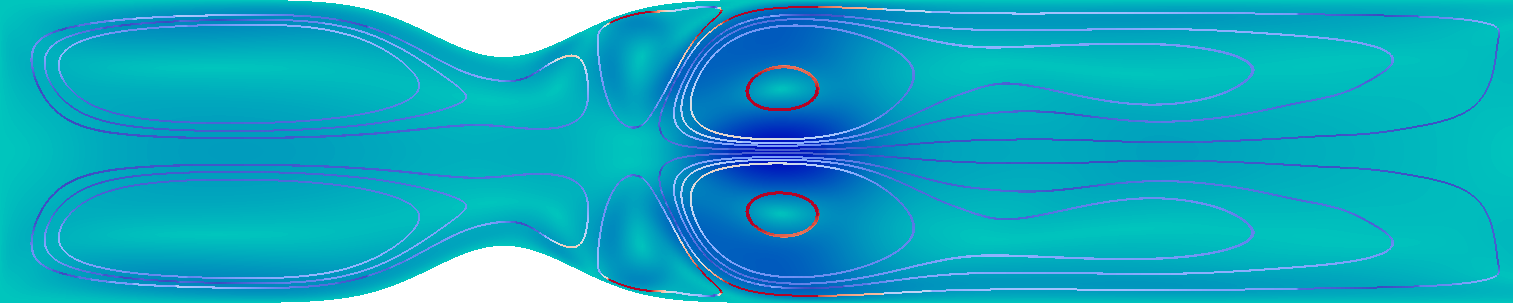}
    \includegraphics[width=0.48\textwidth]{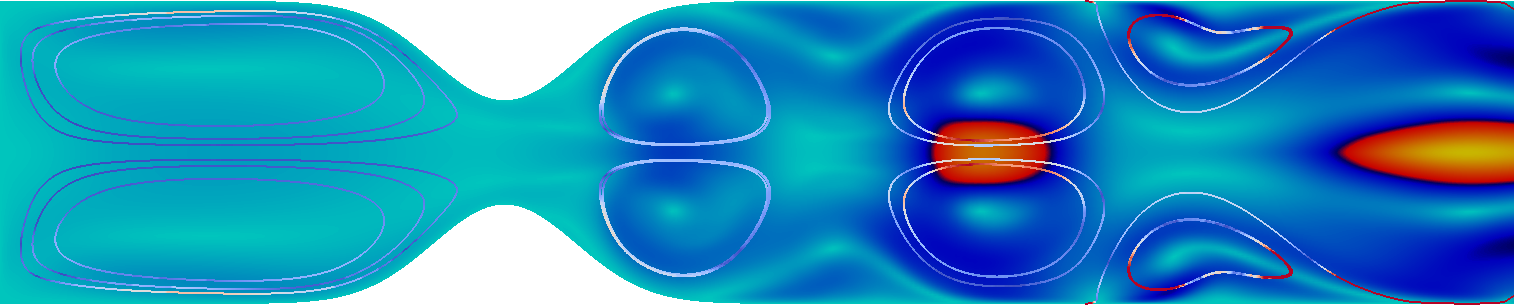}

    \smallskip
    \includegraphics[width=0.48\textwidth]{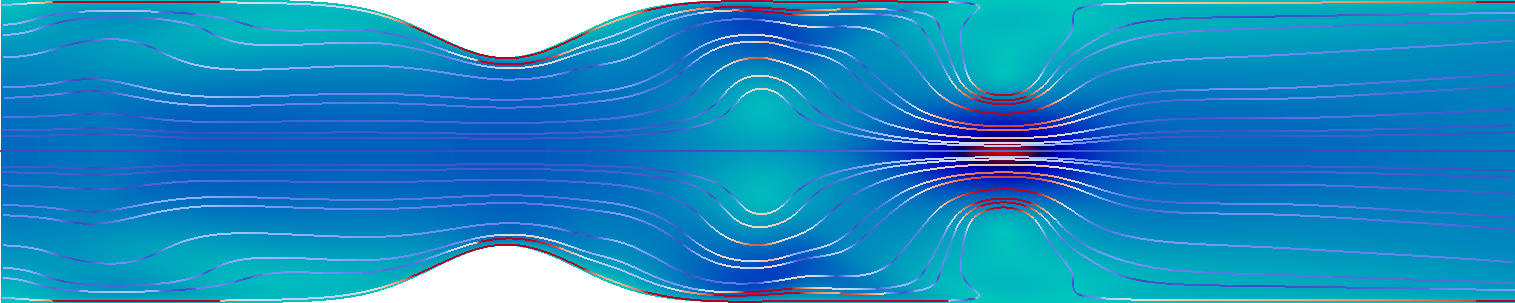}
    \includegraphics[width=0.48\textwidth]{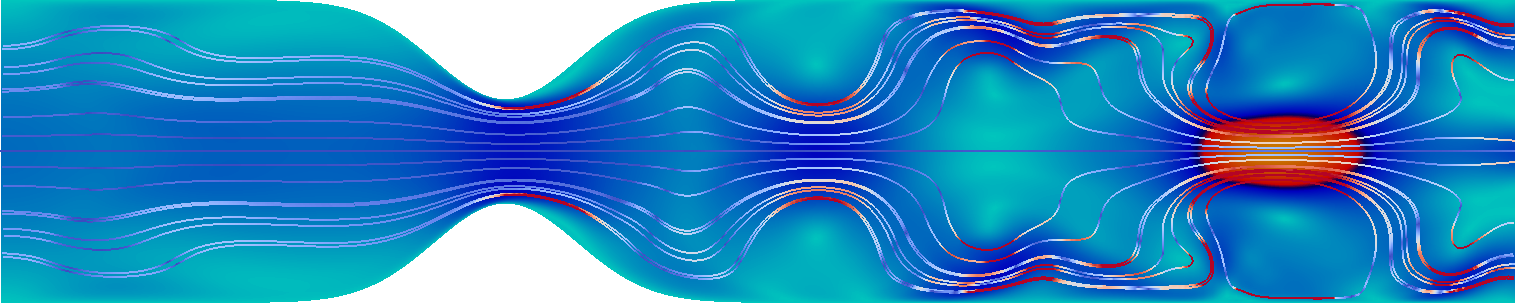}

    \smallskip
    \includegraphics[width=0.48\textwidth]{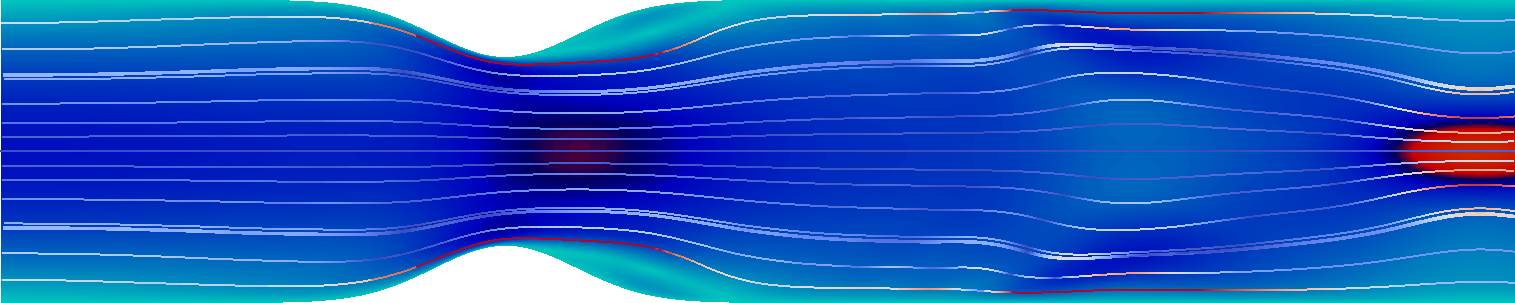}
    \includegraphics[width=0.48\textwidth]{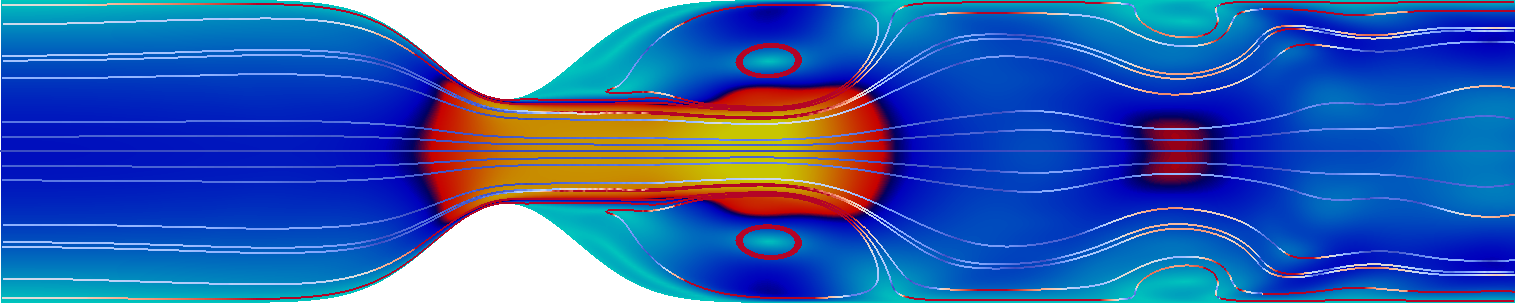}

    \smallskip
    \includegraphics[width=0.48\textwidth]{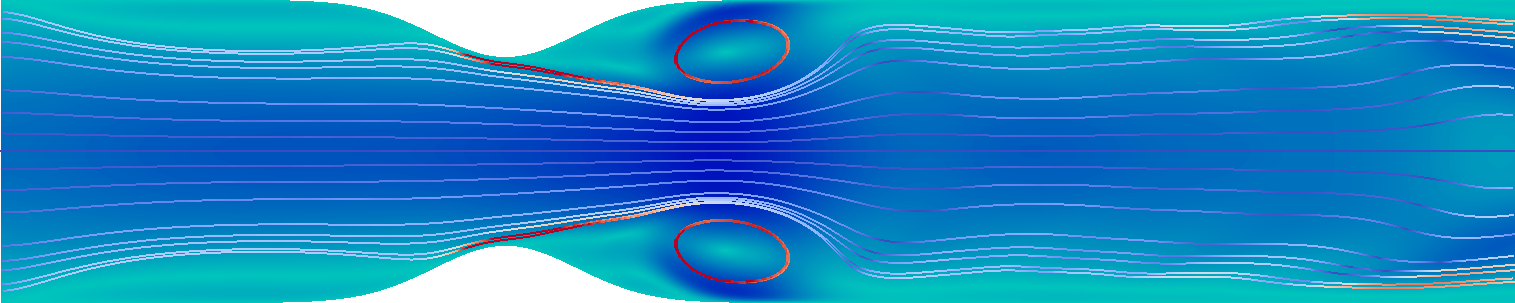}
    \includegraphics[width=0.48\textwidth]{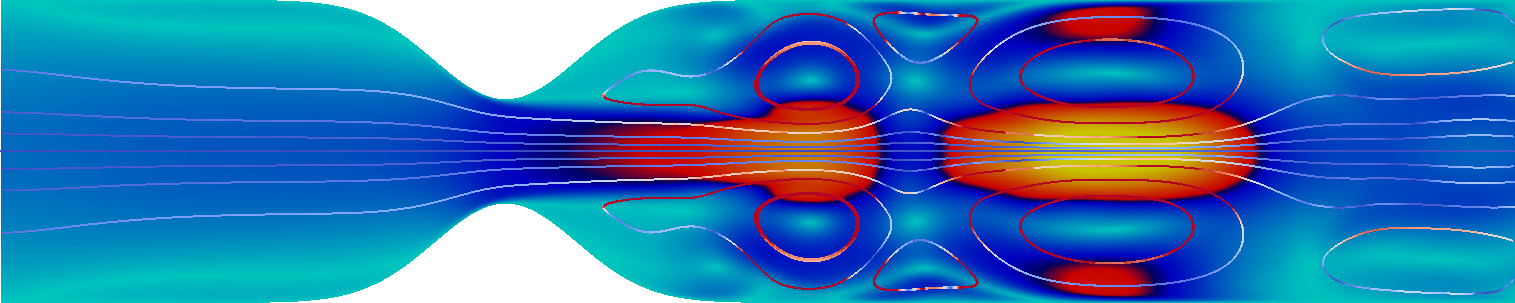}

  \end{center}
  \caption{Velocity magnitude at times $t=n\cdot 6\,250$s for
    $n=1,2,\dots,8$ on domains with different growth. As the inflow
    profile is periodic, the narrowing of the domain causes a significant
    change of the flow pattern. }
  \label{fig:vtks}
\end{figure}

In a first test we take the value
$\epsilon=5\cdot 10^{-5}$ in~(\ref{reaction2}). By this choice, the
concentration $u$ reaches approximately $1$ after about
$\unit[50\,000]{s}\approx \unit[1]{day}$ such that we can still
resolve the coupled problem in all temporal 
scales (although the direct simulation still takes a substantial 
effort). To keep the computational effort within bounds we use a
rather coarse spatial discretisation with $320$ elements, resulting in
$4\,131$ unknowns of a biquadratic equal-order discretisation for
velocities and pressure. 

In Figure~\ref{fig:num1:1} (left) we give an overview of the temporal
evolution of the concentration variable $u_k(t)$ using a full resolution of the 
fast scale. The simulations
break down at $T\approx \unit[55\,000]{s}$ due to the deterioration of
the ALE map and the high Reynolds number. 
For the small interval $[10\,000\unit{s},10\,002\unit{s}]$ we show a
close-up view of the resolved solution $u_k(t)$ and the averaged
value $\int_{t}^{t+1}u_k(s)\,\text{d}s$. The deviation is bound
by $3\cdot 10^{-6}=\O{\epsilon}$, in agreement with
Lemma~\ref{lemma:conterror}.  
We determine reference values $u_{ref}(T_n)$ by extrapolating
numerical results for $k=0.05\unit{s}$, 
$k=0.025\unit{s}$ and $k=0.0125\unit{s}$.  The relative errors 
$|u_k(t_n)-u_{ref}(t_n)|/|u_{ref}(t_n)|$ (based on
these extrapolated 
errors) are given in  Figure~\ref{fig:num1:1} (right). The convergence rate in
terms of $k$ is approximately quadratic. Further, there is no
significant accumulation of simulation errors over time. 

By the evolution of the concentration $u$ over time, the
computational domain undergoes substantial deformations with a strong
narrowing of the flow domain. In Figure~\ref{fig:vtks} we show
snapshots of the solution at different time steps, $t\approx
6250s, 13500s, 18750s,...,50000s$. The narrowing of the gap causes an acceleration of the fluid resulting
in a higher Reynolds number flow with a substantial variation in the
feedback functional $R(u,\vt)$ which depends on the wall shear
stress. 

\paragraph{Multiscale approach}

Next, in Figure~\ref{fig:num1:2} we show the results obtained with the
multiscale method for this relaxed problem with $\epsilon=5\cdot
10^{-5}$. The tolerance for approximating the periodic flow 
problems is set to
\[
\|\vt_U(1)-\vt_U(0)\|_{L^2(\Omega)}^2 +
\|p_U(1)-p_U(0)\|_{L^2(\Omega)}^2  < \tolP^2\coloneqq 10^{-8}. 
\]
For each of the three short time step sizes
$k=0.05$s, $k=0.025$s and $k=0.0125$s we use long time step sizes
ranging from $K=6\,400$s to $K=400$s. In the left plot we compare the
solutions for different values of the long time step size $K$. In this
(non-logarithmic) plot we see convergence of the results to the
corresponding resolved simulation with the same short step size
$k=0.025$s. The lower plot shows the corresponding results for a
variation of the small step size $k$, while the long scale step
size is fixed to $K=400$s. For comparison we show the results obtained
with the resolved simulation for these small-step sizes. Again we see convergence of the
multiscale  scheme towards the resolved scheme.

\begin{table}[t]
  \begin{center}
    \begin{tabular}{l|lll|lll}
      \toprule
      &\multicolumn{3}{c|}{$U_K$}
      &\multicolumn{3}{c}{error (w.r.t. extrapolation)}\\
      $k$ & $0.05$s& $0.025$s& $0.0125$s & $0.05$s& $0.025$s& $0.0125$s\\
      \midrule
      $K=6400$s&0.8089225 &0.8118418 &0.8126760&$1.11\cdot 10^{-2}$&$7.55\cdot 10^{-3}$&$6.53\cdot 10^{-3}$ \\
      $K=3200$s&0.8122441 &0.8153539 &0.8162696&$7.05\cdot 10^{-3}$&$3.25\cdot 10^{-3}$&$2.13\cdot 10^{-3}$ \\
      $K=1600$s&0.8132730 &0.8164325 &0.8173319&$5.80\cdot 10^{-3}$&$1.93\cdot 10^{-3}$&$8.35\cdot 10^{-4}$ \\
      $K=800 $s&0.8135139 &0.8166886 &0.8175426&$5.50\cdot 10^{-3}$&$1.62\cdot 10^{-3}$&$5.77\cdot 10^{-4}$ \\
      $K=400 $s&0.8135782 &0.8167926 &0.8176490&$5.42\cdot 10^{-3}$&$1.49\cdot 10^{-3}$&$4.47\cdot 10^{-4}$ \\
%      $K=200$s&0.8135942 &0.8168135 &0.8176831&$5.40\cdot 10^{-3}$&$1.47\cdot 10^{-3}$&$4.05\cdot 10^{-4}$ \\
%      $K=100$s&0.8135983 &0.8168200 &0.8176907&$5.40\cdot 10^{-3}$&$1.46\cdot 10^{-3}$&$3.96\cdot 10^{-4}$ \\
      \midrule
      resolved &0.8135999 &0.8168226 &0.8176928 & $5.42\cdot 10^{-3}$&$1.46\cdot 10^{-3}$&$3.93\cdot 10^{-4}$\\
      \bottomrule
    \end{tabular}
    \smallskip
    \[
    \begin{aligned}
      \text{Fit to }U(k,K) &= U + C_k k^{q_k} +
      C_K K^{q_K}\qquad   U = 0.818006\pm 10^{-3}\%\\
      C_k=-1.12\pm  17\%,\quad  C_k &=-6.61\cdot 10^{-10}\pm 34\%,\quad
      q_k = 1.85\pm 3.39\%,\quad  q_K = 1.80\pm 2.14\%\\
    \end{aligned}
    \]      
  \end{center}
  \caption{Convergence of the multiscale method at time
    $T=\unit[51\,200]{s}$. We show the values of $U_K$ and the error
    (w.r.t. the extrapolation in $k\to 0$ and $K\to 0$). We compare the
    results of the multiscale method with the fully resolved forward
    computation. Finally, we fit the numerical results to the
    expected convergence behaviour. }
  \label{tab:num1:1}
\end{table}

\begin{figure}[t]
  \begin{center}
    \includegraphics[width=0.48\textwidth]{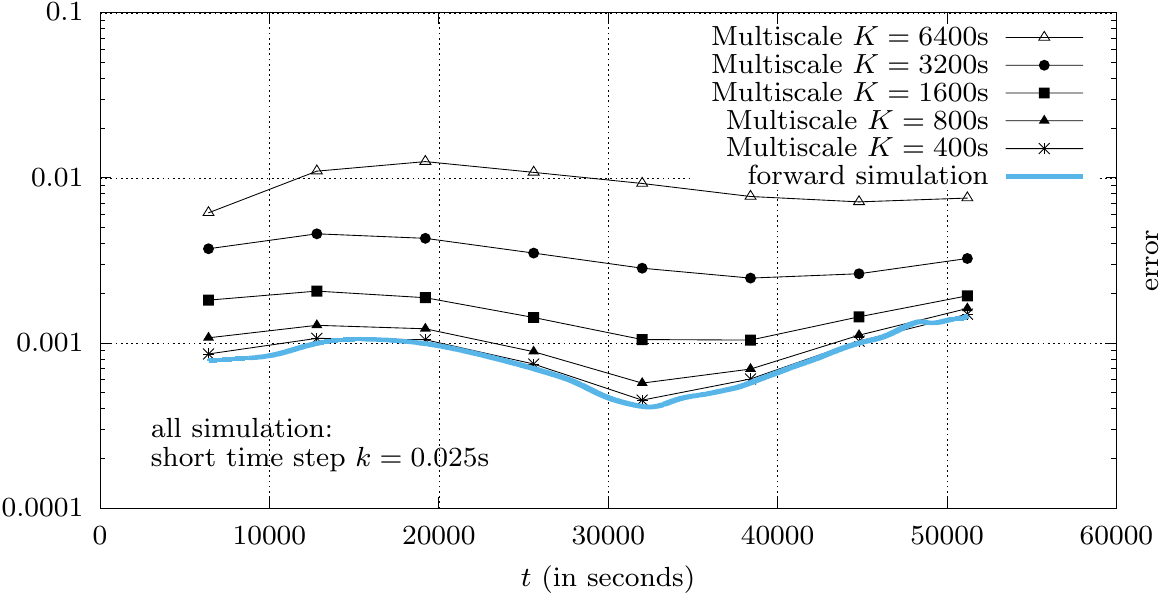}
    \includegraphics[width=0.48\textwidth]{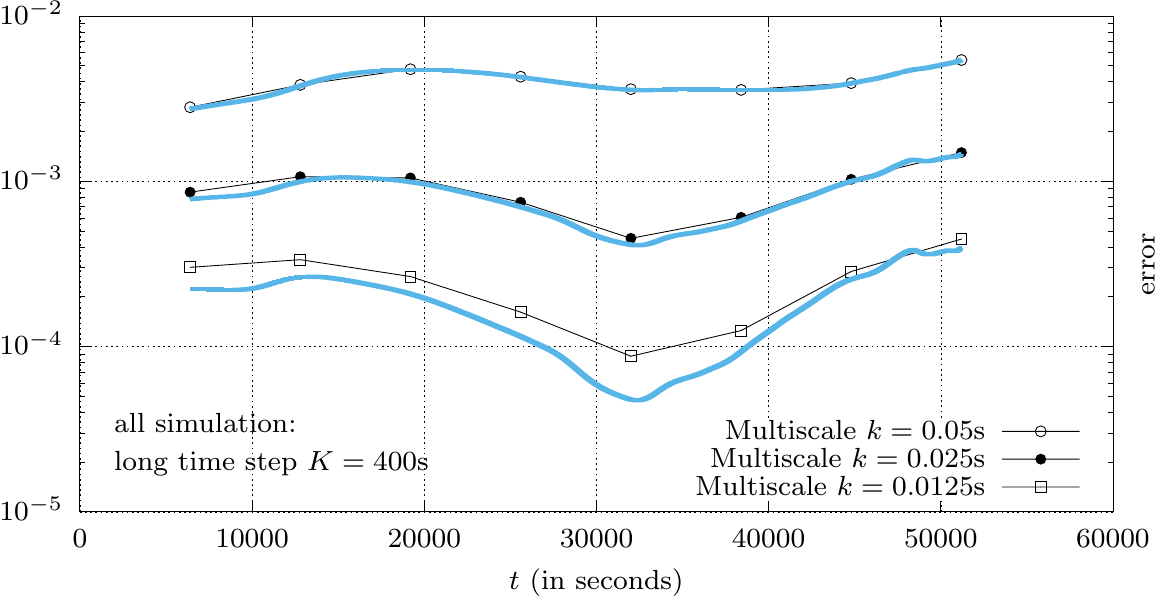}
    \caption{Convergence of the temporal multiscale method for the
      relaxed problem $\epsilon=5\cdot 10^{-5}$.
      \emph{Left}: Effect of the long time step $K$ using the small time step
      $k=0.025$s. \emph{Right}: Effect of the short time step $k$ using the
      long time step size $K=400\unit{s}$.
      For comparison we plot the 
      error of the fully resolved simulation using these time-step sizes. 
      }
    \label{fig:num1:2}
  \end{center}
\end{figure}

The effect of the small step size $k$ is dominant.  This is
highlighted by a closer analysis of the convergence at time
$t=51\,200$s, the results being shown in Table~\ref{tab:num1:1}. We
indicate the concentration $U(t)$ and the errors for the different 
multiscale approaches as well as for the resolved forward
simulation. We fit all these values to the postulated relation
\begin{equation}\label{convergence:kappa}
  U(k,K) = U + C_k k^{q_k} + C_K K^{q_K}
\end{equation}
to get a better understanding of the convergence rates. We estimate
all parameters $u,C_k,C_K,q_k,q_K$ (obtained with gnuplot
fit~\cite{gnuplot})  and find 
\[
U(k,K) = U - 1.12\cdot  k^{1.85} - 6.61\cdot 10^{-10}\cdot
K^{1.80},
\]
see also Table~\ref{tab:num1:1}. 
Convergence is close to the expected second order, both in $k$ and
$K$. The most striking result is the good estimation of the error
constant that shows the proper scaling in  $\epsilon^2$. 
This result
is in good correspondence to the error estimate derived in
Theorem~\ref{thm:main} where the constant in front of the $K^2$-term
depends on $\epsilon^2$. 
Balanced discretisation errors are given for 
$\epsilon^2 K^2 \approx k^2$, i.e. for $K \approx \epsilon^{-1} k$.

Based on the time step relation we can compute the possible speedup
of the multiscale approach which we measure in the overall number of
Navier-Stokes time steps to be performed. 
The forward algorithm requires 
$E_{fwd} = \frac{T}{k}$ solution steps,
while the multiscale approach has an effort of
$E_{ms} = \nicefrac{T}{K}\cdot n_{period} \nicefrac{1}{k} = \nicefrac{T
  n_{period}}{kK}$ steps, where $n_{period}$ is the number of cycles
that are necessary to compute a periodic solution.
%, see Section~\ref{sec:period}. 
Given $K\approx \epsilon^{-1} k$ we approximate
$E_{ms} \approx \nicefrac{\epsilon Tn_{period}}{k^2}$,
and the speedup is estimated by
\[
\frac{E_{fwd}}{E_{ms}} = \frac{k}{\epsilon n_{period} }. 
\]
In our numerical example we identify $n_{period}\le 5$ and with $\epsilon= 
5\cdot 10^{-5}$  we expect a speedup of $4\,000 k$. In 
Figure~\ref{fig:num1:3} we plot the error over the required number of
Navier-Stokes time steps. By circles we indicate the 
multiscale results with a balanced error contribution, which we define
as the state, where the error of the multiscale approach is within $10\%$ of the error
of a fully resolved simulation for the same $k$. We observe speedups of 1:250 for
$k=\unit[0.05]{s}$, 1:180 for $k=\unit[0.025]{s}$ and 1:90 for
$k=\unit[0.0125]{s}$, slightly better values than  the predicted ones
based on $4000 k$. The overall 
computational time for the forward simulation with $k=\unit[0.0125]{s}$ was
about 13 days, while the multiscale simulation with $k=\unit[0.0125]{s}$ and
$K=\unit[800]{s}$, giving a comparable accuracy, was about
$\unit[45]{min}$.

\begin{SCfigure}
  \centering
  \includegraphics[width=0.48\textwidth]{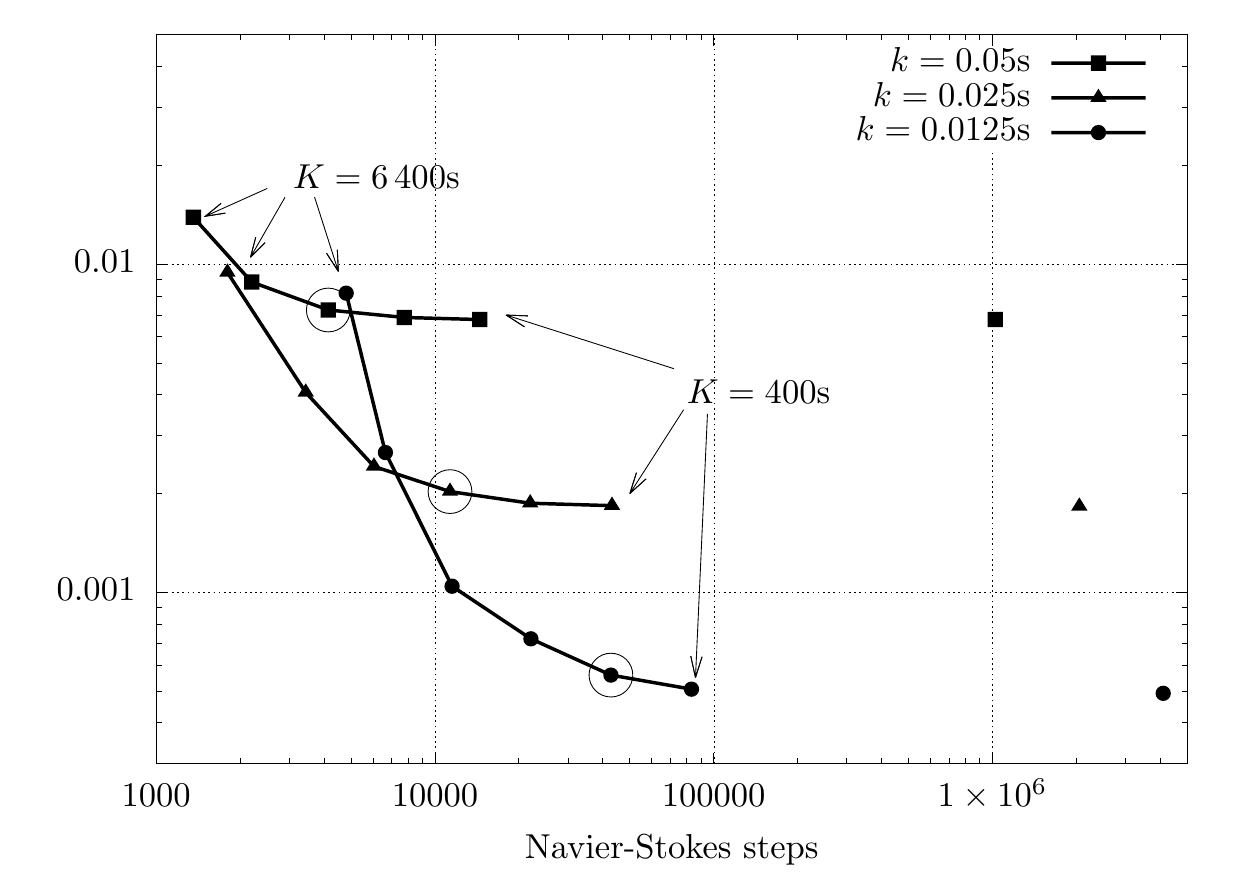}
  \caption{Computational effort (measured in Navier-Stokes times steps)
    for the multiscale approach (lines) and the 
    resolved forward simulation (points). We use three small time
    steps $k$ from $0.05$s down to $0.0125$s and vary the long time
    step $K$ from $6\,400$s to $400$s. Circles indicate multiscale
    solutions of a quality comparable to the resolved forward
    simulation (at most 10\% additional error). The computational time
    for one Navier Stokes step is about $0.2$s (Core i7-7700, 3.60GHz,
    1370 spatial unknowns, biquadratic finite elements).  } 
  \label{fig:num1:3}
\end{SCfigure}

\subsubsection{Configuration with realistic time scales}

Finally, we consider the coupled problem with the time scale parameter
$\epsilon = 10^{-6}$, which is close to the temporal dynamics of 
atherosclerotic plaque growth and
50 times smaller than in the first example. Here,
a resolved forward simulation is not feasible. The concentration
$u(t)$ will reach  a value of  approximately $0.8$ at
$T\approx 2.5\cdot 10^6\unit{s}\approx \unit[30]{days}$. 

\begin{table}[t]
  %  \input{data/LONG/extra.tex}
  % 0.12  a = 0.636257 % 0.06 a = 0.599325 % 0.03 a = 0.592343
\resizebox{\textwidth}{!}{%          
\begin{tabular}{r|rrr|rrr|rrr}
\toprule
\multicolumn{1}{l}{}
&\multicolumn{3}{c|}{$h=0.16$cm}&\multicolumn{3}{c|}{$h=0.08$cm}&\multicolumn{3}{c}{$h=0.04$cm}\\
\multicolumn{1}{l}{}& $k=0.05$s & $k=0.025$s & $k=0.0125$s & $k=0.05$s & $k=0.025$s & $k=0.0125$s & $k=0.05$s & $k=0.025$s & $k=0.0125$s\\
\midrule
$K=204800$s & $ 4.45\cdot 10^{-3} $ & $ 2.46\cdot 10^{-3} $& $
2.20\cdot 10^{-3} $ & $ 4.06\cdot 10^{-3} $& $ 2.42\cdot 10^{-3} $& $
2.04\cdot 10^{-3} $& $ 4.06\cdot 10^{-3} $& $ 2.42\cdot 10^{-3} $&
$ 2.04\cdot 10^{-3} $\\
$K=102400$s & $ 2.86\cdot 10^{-3} $ & $ 8.94\cdot 10^{-4} $& $ 6.17\cdot
10^{-4} $ & $ 2.75\cdot 10^{-3} $& $ 1.06\cdot 10^{-3} $& $
6.28\cdot 10^{-4} $& $ 2.75\cdot 10^{-3} $& $ 1.06\cdot 10^{-3} $&
$ 6.28\cdot 10^{-4} $\\
$K=51200$s & $ 2.43\cdot 10^{-3} $ & $ 4.76\cdot 10^{-4} $& $ 1.96\cdot
10^{-4} $ & $ 2.39\cdot 10^{-3} $& $ 7.04\cdot 10^{-4} $& $
2.58\cdot 10^{-4} $& $ 2.39\cdot 10^{-3} $& $ 7.04\cdot 10^{-4} $&
$ 2.58\cdot 10^{-4} $\\
\midrule
\multicolumn{2}{l}{extrapolated ($k,K\to 0$)}&
\multicolumn{2}{r|}{$|U(T)-U_h(T)| \approx 4.55\cdot
10^{-2}$}&\multicolumn{3}{r|}{$|U(T)-U_h(T)|\approx 8.61\cdot
10^{-3}$}&\multicolumn{3}{r}{$|U(T)-U_h(T)|\approx 1.63\cdot 10^{-3}$}\\
\bottomrule
\end{tabular}}
%  
%  & 0.05 & 0.025 & 0.0125 
%  & $ 4.06e-03 $& $ 2.42e-03 $& $ 2.04e-03 $
%  & $ 2.75e-03 $& $ 1.06e-03 $& $ 6.28e-04 $
%  & $ 2.39e-03 $& $ 7.04e-04 $& $ 2.58e-04 $
%  
%  
%  & 0.05 & 0.025 & 0.0125
%  & $ 3.88e-03 $& $ 2.37e-03 $& $ 1.99e-03 $
%  & $ 2.63e-03 $& $ 1.05e-03 $& $ 6.23e-04 $
%  & $ 2.29e-03 $& $ 7.02e-04 $& $ 2.61e-04 $
%

%
% fit in h:
%
% 0.16 0.636257
% 0.08 0.599325
% 0.04 0.592343
%
% a(h) = 0.590715 + 3.72422 h^2.40316

  \caption{Convergence of the multiscale approach for
    $\epsilon=10^{-6}$. On three mesh levels we indicate the errors in
    the concentration $U(T)$ at $T=\unit[1\,843\,200]{s}\approx \unit[21]{days}$. In each block,
    the error are given w.r.t. the extrapolation $k,K\to 0$. In the
    last line we indicate the (dominating) spatial error for each block. } 
  \label{tab:num2}
\end{table}

Assuming the validity of estimate~(\ref{convergence:kappa}) and in
addition that $C_K\approx \epsilon^2$ we expect balanced error
contributions for $K \approx\epsilon^{-1}k = 10^6 k$. The character of
the short scale problem does not depend on $\epsilon$. Hence we
consider again the step sizes $k=0.05$s, $k=0.025$s and $k=0.0125$s. The
large time step, however, can be significantly increased. We present 
results for $T\approx \unit[21]{days}$ in Table~\ref{tab:num2}. For
this second example, we vary also the mesh size $h$ to discuss  the
impact of all relevant discretisation  parameters. While a smaller
value of $\epsilon$ makes the time scale challenge more severe, the
multiscale approach will profit, as the potential speedup will benefit
from the relation $K\approx \epsilon^{-1} k$.

\begin{figure}[t]
  \begin{center}
    \begin{minipage}{0.4\textwidth}
      \includegraphics[width=\textwidth]{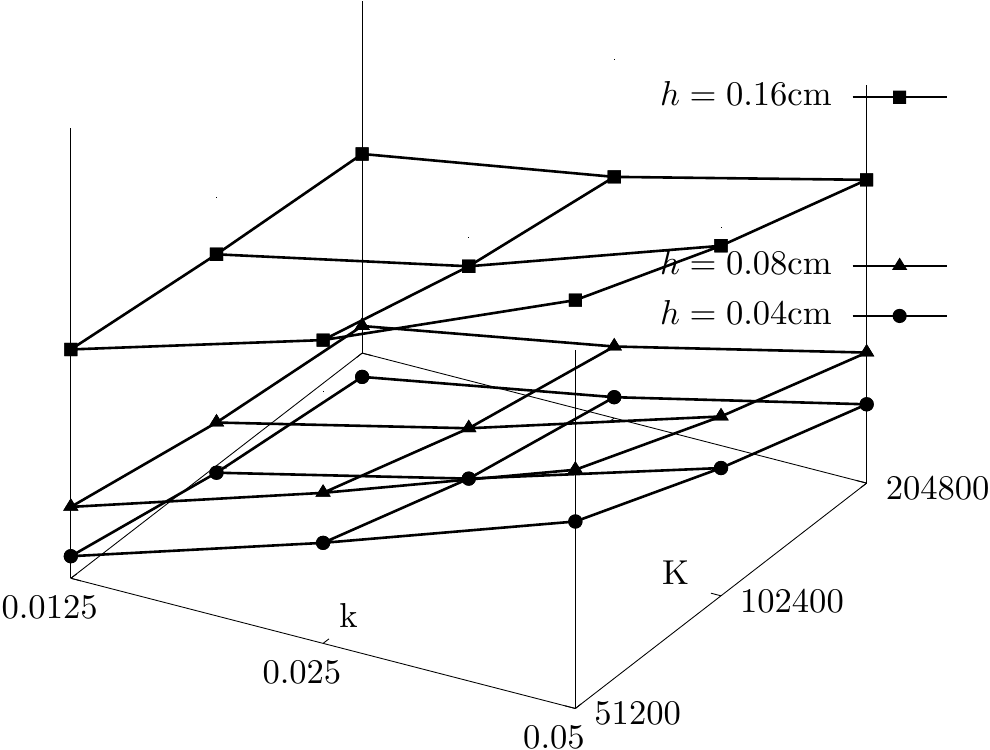}
    \end{minipage}
    \begin{minipage}{0.55\textwidth}
      \includegraphics[width=\textwidth]{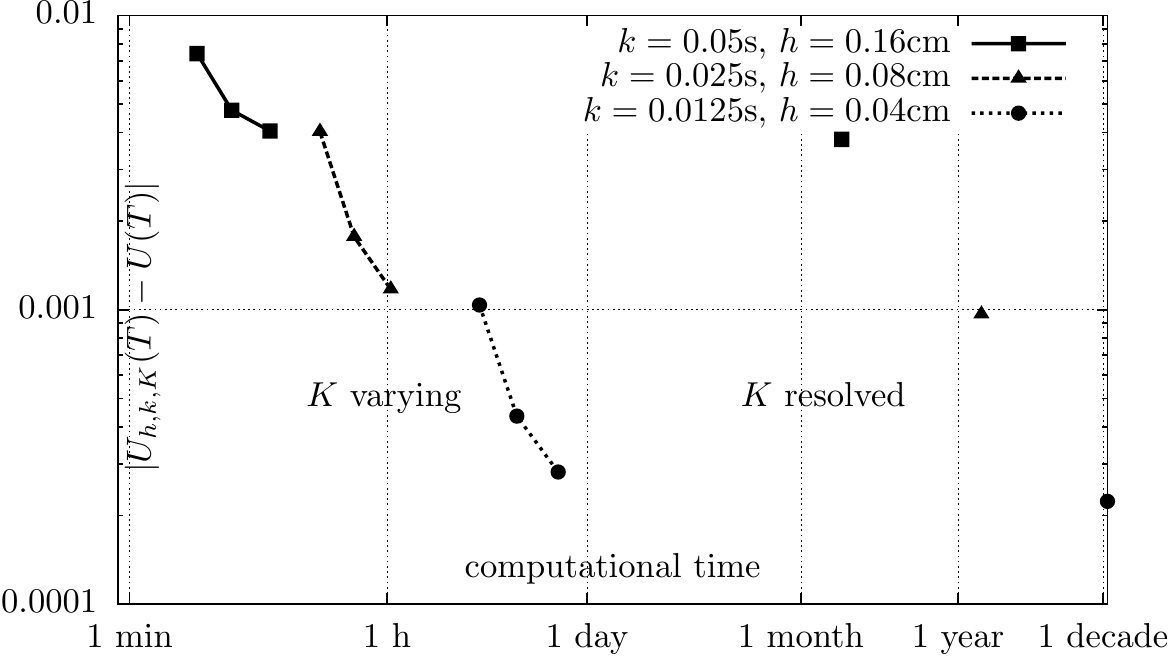}
    \end{minipage}
  \end{center}
  \caption{
    \emph{Left}: Error of the multiscale method under refinement in
    $h,k$ and $K$ versus the extrapolated reference value. \emph{Right}:
    Comparison of the computational times of the multiscale method with
    the corresponding results for a resolved forward simulation. These
    results are based on an extrapolation of the error and a prediction of the
    computational times by multiplying the number of required time steps
    with the average computing times for each step.}
  \label{fig:num2}
\end{figure}

Combining all 27 computations based on three values for $h,k$ and $K$
we find the relation
\[
U(h,k,K)\approx 0.59076+
7.6 h^{2.4} - 1.7 k^{2.2} - 0.04 \epsilon^2 K^{1.9},
\]
which shows approximately second order convergence in both time step
sizes and the mesh size and also the proper scaling of the constants
in the $\O{k^2}$ and $\O{K^2}$ terms. Spatial and temporal errors show
a different sign which is also seen in Figure~\ref{fig:num2}, where we
plot the errors for all computations. In the right sketch of this
figure we compare the computational times of the multiscale approach
with a hypothetical resolved simulation. Here, the errors are
predicted by extrapolation. 
%and use of~(\ref{})
The computational
times are based on the number of Navier Stokes steps, namely
$k^{-1} T$ and the average computational time for each Navier
Stokes step, which is  $\unit[0.135]{s}$ on the  $h=\unit[0.16]{cm}$
mesh, $\unit[0.62]{s}$ for $h=\unit[0.08]{cm}$ and 
$\unit[2.3]{s}$ for $h=\unit[0.04]{cm}$. The results are very similar to those shown in
Figure~\ref{fig:num1:3} for the first example. The best multiscale
results are close to the hypothetical resolved results. Here however,
the savings are substantially larger, with 10 minutes vs. 2 months
(factor 1:8000) for $h=\unit[0.16]{cm}$, 1 hour vs. nearly 2 years for
$h=\unit[0.08]{cm}$ (factor 1:12000) and
15 hours vs. more than 10 years for $h=\unit[0.04]{cm}$ (factor 1:6000). 

Finally, we also evaluate the effect of the parameter $\tolP$ used
to control the periodicity of the Navier-Stokes solution, compare
Theorem~\ref{thm:main}. In Table~\ref{tab:num2:2} we show the errors
at $T=\unit[1\,843\,200]{s}\approx \unit[1]{month}$ for computations
based on $K=\unit[25\,600]{s}\approx \unit[7]{h}$, 
$k=\unit[0.0125]{s}$ and $h=\unit[0.08]{cm}$. The effect of $\tolP$
is very small.

\begin{table}[t]
  \begin{center}
    \begin{tabular}{c|cccc}
      \toprule
      $\tolP$&$10^{-1}$& $10^{-2}$& $10^{-3}$& $10^{-4}$\\
      \midrule
      $\big|U_{h,k,K}(T)_{|_{\tolP}}-U_{h,k,K}(T)_{|_{\tolP=10^{-8}}}  
      \big|$&$1.99\cdot 
      10^{-5}$&$4.88\cdot 10^{-7}$&$3.19\cdot 10^{-7}$&$1.22\cdot 10^{-7}$ \\ 
      \bottomrule
    \end{tabular}
  \end{center}
  \caption{Impact of the periodicity parameter $\tolP$ on the
    error in concentration $U$ in $T=1\,843\,200$s. Computed with respect to $\tolP=10^{-8}$.
    The discretisation is chosen as $h=0.08$cm, $k=0.0125$s and
    $K=25\,600$s.}
  \label{tab:num2:2}
\end{table}

\section{Conclusion}

We have presented a framework for the simulation of temporal
multiscale problems, where we are interested in the evolution of a
slow variable which depends on an oscillating fast variable. The
numerical schemes are designed for models that are given by
partial differential equations. The most important assumption is a
local (in time) proximity of the fast scale variable to the
solution of a periodic problem. An effective scheme for the slow
variable is derived by replacing the fast variable with the periodic
solution which can be computed locally, as no initial values must be
transferred. The only overhead of the multiscale scheme comes from the
identification of initial values required for approximating the
periodic problems. Nevertheless, we gain huge speedups compared to a
simulation with resolved time scales. The efficiency of the multiscale
approach increases when the time scale separation gets larger. 

The resulting scheme depends on several numerical parameters, small
and large time steps $k$ and $K$, and the spatial mesh size $h$. It
remains a topic for a future work to design an automatic and adaptive
algorithm to control all these parameters in order to balance
all contributing error terms.

%\bibliographystyle{plain}
%\bibliography{lit}

\end{document}